\documentclass[a4size, 12pt]{article}

\usepackage[top = 1.3in, bottom = 1.3in, left = 0.80in, right = 0.80in]{geometry}
\usepackage{longtable, tabu}
\usepackage{amscd, amsfonts, amsmath, amssymb, amsthm, animate, array, arydshln, chngcntr, color, float, graphicx, hyperref, mathrsfs, multirow, latexsym, lscape, MnSymbol, geometry, pdflscape, ragged2e, soul, tabu, textcomp, ulem, url, wrapfig}

\theoremstyle{theorem}\newtheorem{theorem}{Theorem}[section]
\theoremstyle{theorem}\newtheorem{lemma}[theorem]{Lemma}
\theoremstyle{theorem}\newtheorem{proposition}[theorem]{Proposition}
\theoremstyle{theorem}\newtheorem{corollary}[theorem]{Corollary}
\theoremstyle{definition}\newtheorem{definition}{Definition}[section]
\theoremstyle{definition}
\theoremstyle{definition}
\theoremstyle{definition}
\theoremstyle{definition}
\counterwithin{table}{section}

\begin{document}
    \sloppy
    \begin{center}
        \large{\textbf{String C-groups from groups of order $2^m$ and exponent at least $2^{m - 3}$}} \\[4mm]
    \end{center}

    \begin{center}
        \textbf{Mark L. Loyola} \\[1mm]

        mloyola@ateneo.edu \\
        Department of Mathematics \\
        Ateneo de Manila University, Philippines
    \end{center}

    \begin{abstract}
        This work provides a classification of string C-groups of order $2^m$ and exponent at least $2^{m - 3}$. Prior to the classification, we complete the list of groups of exponent $2^{m - 3}$ and rank at least 3. The main result states that, aside from the cyclic and dihedral groups, only two groups of exponent at least $2^{m - 3}$ are connected string C-groups. Both have rank 3 and are quotients of the string Coxeter group $W = [4, 2^{m - 3}]$.
    \end{abstract}

\section{Introduction}
\label{sec:intro}

    A set $S = \{s_0, s_1, \ldots, s_{n - 1}\}$, $n \geq 1$, of distinct involutions in a group $\mathcal{G}$ (not necessarily finite or a 2-group) can be assigned an edge-labeled graph or diagram $\mathcal{D}(S)$ consisting of $n$ vertices. Each vertex of $\mathcal{D}(S)$  represents an involution and an edge connects the vertices representing $s_j$ and $s_k$ and is labeled $p_{j, k} = \text{ord}(s_js_k)$ if and only if $p_{j, k} \geq 3$.  By convention, we omit the label in the diagram when $p_{j, k} = 3$. If $p_{j, k} = 2$ for $0 \leq j, k \leq n - 1$ whenever $|j - k| \geq 2$, then $\mathcal{D}(S)$ is a \textit{string diagram}. We call this property the \textit{string condition}. If, in addition to the string condition, we also have $p_{j, j + 1} > 2$ for $0 \leq j \leq n - 2$, then $\mathcal{D}(S)$ is a \textit{connected string diagram} (see \textbf{Figure \ref{fig:stringdiag}}).
    \begin{figure}[H]
        \centering
        \includegraphics[scale = 0.33, keepaspectratio = true]{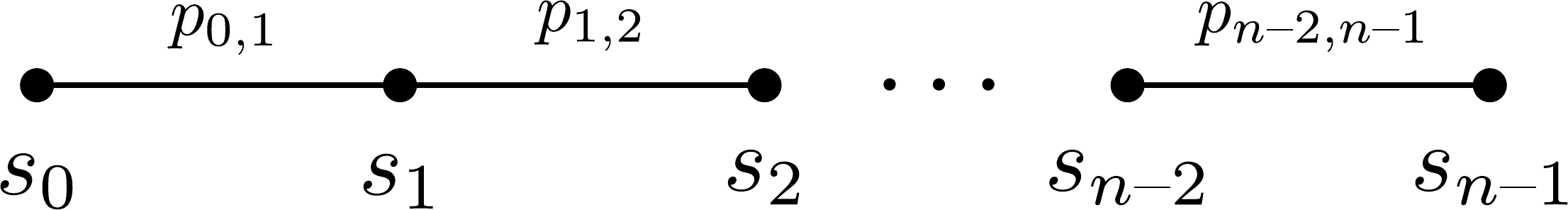}
        \caption{A connected string diagram.}
        \label{fig:stringdiag}
    \end{figure}

    The following terms will be used to describe the pair $(\mathcal{G}, S)$ in the special case when $\mathcal{G} = \langle{S}\rangle$.
    \begin{definition}[String C-Group]
        Let $\mathcal{G}$ be a group generated by a set $S = \{s_0, s_1, \ldots, s_{n - 1}\}$, $n \geq 1$, of involutions called its \textit{distinguished generators}. The pair $(\mathcal{G}, S)$ is called a \textit{string C-group} with respect to $S$ if $\mathcal{D}(S)$ is a string diagram and $S$ satisfies the property that
        \[ \text{for all $J, K \subseteq \{0, 1, \ldots, n - 1\}$, } \langle{s_j \mid j \in J}\rangle \cap \langle{s_k \mid k \in K}\rangle = \langle{s_i \mid i \in J \cap K}\rangle. \]
        The latter property is called the \textit{intersection condition}. A string C-group $(\mathcal{G}, S)$ is called a \textit{connected string C-group} if $\mathcal{D}(S)$ is a connected string diagram.
        \label{def:StringCGroups}
    \end{definition}

    It is important to remark that a string C-group $(\mathcal{G}, S)$ always involves a pair consisting of a group $\mathcal{G}$ and a set $S$ of distinguished generators. If we refer simply to $\mathcal{G}$ as a string C-group, then it is implied that either a set $S$ of distinguished generators for $\mathcal{G}$ has been specified or that we can find a suitable set $S$ of generating involutions that satisfies the intersection condition.

    The definition of a string C-group stemmed from the extension of the concept of a regular polytope from classical geometry to an \textit{abstract polytope}. Whereas classical regular polytopes include, for instance, regular polyhedra and regular star polyhedra in the Euclidean space, abstract polytopes include certain abstract posets, regular maps, and even tessellations. One of the fundamental results of the theory of abstract regular polytopes is the correspondence between regular abstract polytopes and string C-groups acting as their groups of automorphisms. The book \textit{Abstract Regular Polytopes}~\cite{McMullenSchulte2002} by P. McMullen and E. Schulte offers a comprehensive account of the theory of abstract polytopes and string C-groups including history and recent developments in the subject.

    As with most specialized groups, no known general method exists for checking if a given group is a string C-group. Consequently, most papers that deal with characterizations or classifications of string C-groups consider particular types of groups or families of groups that satisfy certain qualifications or possess interesting properties. These include, for instance, simple and almost simple groups~\cite{HartleyHulpke2010, Leemans2006, LeemansVauthier2006}, symmetric and alternating groups~\cite{FernandesLeemans2011, FernandesLeemansMixer2012}, discrete groups of transformations~\cite{JohnsonWeiss1999, Monson1995}, groups of a fixed rank~\cite{McMullenSchulte1992}, and small groups~\cite{Hartley2006}. The last reference pertains to M.I. Hartley's \textit{Atlas of Small Regular Polytopes}, a web-based library of string C-groups of order at most 2000 (except 1024 and 1536) and their associated regular polytopes. The library is accessible through the link \url{http://www.abstract-polytopes.com/atlas/}. The use of computational group theory software such as GAP~\cite{GAP2015} and Magma~\cite{Magma2015} has been indispensable in these classifications. Such software are extremely helpful most especially in handling groups with a relatively large set of involutions.

    In \textit{Problems on polytopes, their groups, and realizations}~\cite{SchulteWeiss2006}, E. Schulte and A. Weiss proposed the problem of characterizing groups of order $2^m$ or $2^mp$, $p$ an odd prime, which are string C-groups or are automorphism groups of chiral polytopes. This problem arose from the difficulty encountered in the creation of the \textit{Atlas} for the groups of order $2^9 = 512$, $2^{10} = 1024$, and $2^9 \cdot 3 = 1536$. The numbers of groups of the said orders are huge and, without a usable characterization theorem, handling them requires an extensive use of computational resources.

    We contribute to the solution of this problem for a special class of 2-groups. In particular, we are interested in classifying string C-groups of order $2^m$ and exponent at least $2^{m - 3}$. The method of classification we will present takes advantage of the fact that these 2-groups have been previously enumerated, albeit incomplete for groups of exponent exactly $2^{m - 3}$~\cite{Burnside1911, McKelden1906, Ninomiya1994}.

    We outline the paper as follows. In the next section, we discuss the fundamental concepts and results related to 2-groups and string C-groups. We also introduce the notations used throughout this paper. In \textbf{Section \ref{sec:allgroups}}, we give the complete list of groups of order $2^m$ and exponent $2^{m - 2}$ or $2^{m - 3}$ that are relevant to our study. Further, we compute these groups' respective sets of involutions. Finally, in \textbf{Section \ref{sec:stringC}}, we state and prove our main result pertaining to the classification of string C-groups of order $2^m$ and exponent at least $2^{m - 3}$.
    
    \section{Notations and preliminaries}
\label{sec:intro}
    \subsection{$2$-groups}
    \label{subsec:2groups}

    We assume throughout this paper, unless stated otherwise, that $\mathcal{G}$ is a finite 2-group of order $2^m$, where $m$ is a positive integer.  We denote by $\text{exp}(\mathcal{G})$ and $\text{rank}(\mathcal{G})$ the \textit{exponent} and \textit{rank}, respectively, of $\mathcal{G}$. We use $\text{ord}(g)$ to denote the order of an element $g \in \mathcal{G}$.

    By exponent\footnote{Note that for an arbitrary group, the exponent is defined to be the least common multiple of the orders of all elements in the group. If there is no least common multiple, the exponent is taken to be infinity.} of a 2-group $\mathcal{G}$, we mean the maximum of the orders of all the elements in $\mathcal{G}$: $\text{exp}(\mathcal{G}) = \max\{\text{ord}(g) : g \in \mathcal{G}\}$. Thus, $\mathbf{Z}_{\text{exp}(\mathcal{G})}$ is the cyclic group of the largest possible order inside $\mathcal{G}$.

    The \textit{rank} of $\mathcal{G}$ refers to the minimum of the cardinalities of all generating sets for $\mathcal{G}$: $\text{rank}(\mathcal{G}) = \min\{|X| : X \subseteq \mathcal{G}, \langle{X}\rangle = \mathcal{G}\}$.

    By Burnside's Basis Theorem~\cite{Berkovich2008}, any two minimal generating sets for $\mathcal{G}$ have the same cardinality equal to $\text{rank}(\mathcal{G})$. In addition, if $\Phi(\mathcal{G})$ is the \textit{Frattini subgroup} of $\mathcal{G}$ (the intersection of all maximal subgroups of $\mathcal{G}$), then the quotient group $\mathcal{G}/\Phi(\mathcal{G})$ is isomorphic to the elementary abelian group $\mathbf{Z}_2^{\text{rank}(\mathcal{G})}$, the direct product of $\text{rank}(\mathcal{G})$ copies of $\mathbf{Z}_2$.

    Let $S = \{g_0, g_1, \ldots, g_{n - 1}\}$ be a minimal generating set for $\mathcal{G}$ and define $\rho : \mathcal{G} \to \mathcal{G}/\Phi(\mathcal{G})$ to be the canonical projection sending $g_i$ to $\overline{g_i} = g_i\Phi(\mathcal{G})$. Then $\overline{S} = \{\overline{g_0}, \overline{g_1}, \ldots, \overline{g_{n - 1}}\}$ is a minimal generating set for $\mathcal{G}/\Phi(\mathcal{G})$. For any non-empty $\overline{X} \subseteq \overline{S}$, define the map $\psi_{\: \overline{X}} : \mathcal{G}/\Phi(\mathcal{G}) = \langle{\overline{S}}\rangle \to \mathbf{Z}_2$ which sends $\overline{g_i}$ to
    \[
        \psi_{\: \overline{X}}(\overline{g_i}) =
        \begin{cases}
            -1 & \text{ if $\overline{g_i} \in \overline{X}$} \\
            \;\:\: 1 & \text{ if $\overline{g_i} \notin \overline{X}$}.
        \end{cases}
    \]
    Then $\psi_{\: \overline{X}}$ defines a surjective group homomorphism. Consequently, the composition $\varphi_X = (\psi_{\: \overline{X}} \circ \rho)$ also defines a surjective homomorphism from $\mathcal{G}$ to $\mathbf{Z}_2$ for any non-empty subset $X \subseteq S$. We have thus proved the following proposition.
   \begin{proposition}
        Let $S = \{g_0, g_1, \ldots, g_{n - 1}\}$ be a minimal generating set for $\mathcal{G}$. Then for any non-empty subset $X \subseteq S$, the map $\varphi_X : \mathcal{G} \to \mathbf{Z}_2$ sending a generator $g_i \in S$ to
        \[
            \varphi_X(g_i) =
            \begin{cases}
                -1 & \text{ if $g_i \in X$} \\
                \;\:\: 1 & \text{ if $g_i \notin X$} \\
            \end{cases}
        \]
        defines a surjective group homomorphism. \\
        \label{prop:2-groupHom}
    \end{proposition}

    The next result whose proof immediately follows from \textbf{Proposition \ref{prop:2-groupHom}} provides a necessary condition for a 2-group to be generated by its set of \textit{involutions} (elements of order 2): $\text{inv}(\mathcal{G}) = \{g \in \mathcal{G} : \text{ord}(\mathcal{G}) = 2\}$. This corollary will be used later in \textbf{Section \ref{sec:allgroups}} to show that a given 2-group is not generated by involutions.
    \begin{corollary}
        Let $\mathcal{G}$ be a 2-group with minimal generating set $S$. If $\mathcal{G} = \langle{\text{inv}(\mathcal{G})}\rangle$, then for every non-empty subset $X \subseteq S$, $\varphi_X(\langle{\text{inv}(\mathcal{G})}\rangle) = \{-1, 1\}$.
        \label{cor:notgenbyinvs}
    \end{corollary}

    \subsection{String C-groups}
    \label{subsec:stringCgroups}

    String C-groups of rank 2 are the dihedral groups. Hence, in any classification theorem involving string C-groups, one may work with the assumption that a string C-group $(\mathcal{G}, S)$ has rank at least 3. In addition, we can further assume that $\mathcal{D}(S)$ is connected, since any disconnected string C-group may be expressed as a direct product of subgroups corresponding to the connected components of $\mathcal{D}(S)$. These subgroups, which are generated by distinguished generators and which arise from the induced subdiagrams of $\mathcal{D}(S)$, are called \textit{distinguished subgroups} of $\mathcal{G}$.

    Examples of string C-groups are symmetry groups of regular convex and star-polytopes, automorphism groups of regular abstract polytopes~\cite{McMullenSchulte2002}, and string Coxeter groups~\cite{Humphreys1990}. In fact, the ``C'' in the term \textit{string C-group} stands for ``Coxeter.'' The connection between string C-groups and string Coxeter groups is made much deeper and precise in the following statement: a string C-group $\mathcal{G}$ is a \textit{smooth} quotient of the string Coxeter group $W = [p_{0, 1}, p_{1, 2}, \ldots, p_{n - 2, n - 1}]$, the Coxeter group with the same diagram as $\mathcal{G}$. By smooth, we mean that $\text{ord}(s_js_k) = p_{j, k}$ and, thus, a group presentation for $\mathcal{G}$ includes all relations in the presentation for $W$. In this case, we say the $\mathcal{G}$ has \textit{(Schl\"{a}fli) type} $\{p_{0, 1}, p_{1, 2}, \ldots, p_{n - 2, n - 1}\}$.

    Between the two conditions which defines a string C-group, it is the intersection condition that significantly reduces the number of string C-groups from the large number of groups generated by involutions. It is, however, also the condition that is more difficult to verify especially for groups of high ranks. This is mainly because there is no known general method for establishing this property for an arbitrary group. In addition, the number of possible combinations for pairs of subsets of $S$ becomes large as the rank gets larger. A method which uses induction on $\text{rank}(\mathcal{G})$ provided by Conder and Oliveros in~\cite{ConderOliveros2013} reduces the number of cases to be checked when establishing the intersection condition. For example, when $\text{rank}(\mathcal{G}) = 3$, there is only one case to check:
    \begin{lemma}
        The intersection condition is satisfied by the group $\mathcal{G} = \langle{s_0, s_1, s_2}\rangle$ of rank 3 if $\langle{s_0, s_1}\rangle \cap \langle{s_1, s_2}\rangle = \langle{s_1}\rangle$.
        \label{lem:rank3intcond}
    \end{lemma}

    From the \textbf{SmallGroups} Library of GAP, we can perform a search to list down the groups of order $2^m$ and exponent $2^{m - 2}$ for $m = 3, 4, 5$, or exponent $2^{m - 3}$ for $m = 4, 5, 6$. We can then check which among the groups in this list appear in the \textit{Atlas of Small Regular Polytopes} and, hence, are string C-groups. We have the following classification of string C-groups for these small values of $m$:
    \begin{itemize}
        \item For $m = 3, 4, 5$, there are no connected string C-groups of exponent $2^{m - 2}$.
        \item For $m = 4$, there are no connected string C-groups of exponent $2^{m - 3}$.
        \item For $m = 5$, there exists only one connected string C-group of exponent $2^{m - 3}$. This group has GAP ID \texttt{[ 32, 27 ]} and has type $\{4, 4\}$. It is isomorphic to $(\mathbf{D}_2 \times \mathbf{D}_2) \rtimes \mathbf{Z_2}$.
        \item For $m = 6$, there exist two connected string C-groups of exponent $2^{m - 3}$. These groups have GAP IDs \texttt{[ 64, 128 ]} and \texttt{[ 64, 134 ]} and have type $\{4, 8\}$. They are isomorphic to $(\mathbf{D}_4 \times \mathbf{D}_2) \rtimes \mathbf{Z_2}$ and $(\mathbf{D}_4 \rtimes \mathbf{D}_2) \rtimes \mathbf{Z_2}$, respectively.
    \end{itemize}

    Thus, for a complete classification of string C-groups of order $2^m$ and exponent at least $2^{m - 3}$, it suffices to focus on the groups $\mathcal{G}$ of rank at least 3 that satisfy $\text{exp}(\mathcal{G}) = 2^{m}$ for $m \geq 1$, $\text{exp}(\mathcal{G}) = 2^{m - 1}$ for $m \geq 2$, $\text{exp}(\mathcal{G}) = 2^{m - 2}$ for $m \geq 6$, or $\text{exp}(\mathcal{G}) = 2^{m - 3}$ for $m \geq 7$. A complete list of the groups in the last two families is given in the next section.
    
    \section{2-groups of exponent $2^{m - 2}$ ($m \geq 6$) and $2^{m - 3}$ ($m \geq 7$)}
\label{sec:allgroups}

    The determination of groups of order $2^m$ containing cyclic subgroups of maximal order $2^{m - 1}$, $2^{m - 2}$, and $2^{m - 3}$ was (re)accomplished, respectively, by W. Burnside~\cite{Burnside1911}; S. Bai, Burnside, G. Miller, Y. Ninomiya, ~\cite{Bai1985, Burnside1911, Miller1901, Miller1902, Ninomiya1994}; and A. McKelden~\cite{McKelden1906}. Q. Zhang and P. Li in~\cite{ZhangLi2012} asserted, however, that McKelden's list is incomplete and contains some errors. A quick comparison of the number of groups in~\cite{McKelden1906} and in GAP's \textbf{SmallGroups} Library~\cite{GAP2015} for a few values of $m$ gives weight to this assertion. Since the method of proof that we shall employ to classify string C-groups of order $2^m$ and exponent at least $2^{m - 3}$ relies heavily on the completeness of the list, it becomes crucial to fill in the missing groups of exponent $2^{m - 3}$ that are relevant to our study. We run GAP to perform a brute force parameter search to identify these missing groups. In the process, we also amend the tuples of parameters defining some of the groups in McKelden's paper.
    \subsection{$\text{exp}(\mathcal{G}) = 2^{m - 2}$, $m \geq 6$}
    \label{subsec:2^m-2}

    Assume $m \geq 6$. Let $p$ be an element of $\mathcal{G}$ with $\text{ord}(p) = 2^{m - 2}$. Then $[\mathcal{G} : \langle{p}\rangle] = 4$ and $g^4 \in \langle{p}\rangle$ for any $g \in \mathcal{G}$. We consider two possibilities.
    \begin{itemize}
        \item Suppose there exists $q \in \mathcal{G} - \langle{p}\rangle$ such that $q^2 \notin \langle{p}\rangle$. Then $q^i\langle{p}\rangle$ for $0 \leq i \leq 3$ are the four distinct cosets of $\langle{p}\rangle$ in $\mathcal{G}$. It follows immediately that $\mathcal{G} = \langle{p, q}\rangle$ and, therefore, $\text{rank}(\mathcal{G}) = 2$. We denote by $\mathcal{N}_{I}(m)$ the class to which these groups belong.
        \item Suppose that $g^2 \in \langle{p}\rangle$ for any $g \in \mathcal{G}$. Let $q \in \mathcal{G} - \langle{p}\rangle$. Then $\langle{p}\rangle \trianglelefteq \langle{p, q}\rangle$, since $(pq^{-1})^2\langle{p}\rangle = \langle{p}\rangle$ implies that $qpq^{-1}\langle{p}\rangle = q^2p^{-1}\langle{p}\rangle = \langle{p}\rangle$. So every element of $\langle{p, q}\rangle$ can be written uniquely as $p^iq^j$, where $0 \leq i \leq 2^{m - 2} - 1$ and $j = 0, 1$. It follows that $|\langle{p, q}\rangle| = 2^{m - 1}$ and there exists $r \in \mathcal{G} - \langle{p, q}\rangle$ such that $\mathcal{G} = \langle{p, q, r}\rangle$. Therefore, $\text{rank}(\mathcal{G}) \leq 3$. We denote by $\mathcal{N}_{II}(m)$ the class to which these groups belong.
    \end{itemize}

    \subsubsection{Class $\mathcal{N}_I(m)$}
    \label{subsec:ClassN_I}

    There are 18 non-isomorphic groups in class $\mathcal{N}_{I}(m)$. A.M. McKelden in~\cite{McKelden1906} compiled a list of groups in this class and further subdivided them into 7 subclasses, which we denote by $\mathcal{N}_{I-A}(m)$, $\mathcal{N}_{I-B}(m)$, $\mathcal{N}_{I-C}(m)$, $\mathcal{N}_{I-D}(m)$, $\mathcal{N}_{I-E}(m)$, $\mathcal{N}_{I-F}(m)$, and $\mathcal{N}_{I-G}(m)$. The list, however, contained two errors: a pair of isomorphic groups in $\mathcal{N}_{I-D}(m)$ were listed as distinct groups and the group with presentation
    \[
        \Bigg\langle
        p, q \; \Bigg\mid
        \begin{array}{cc}
            p^{2^{m - 2}} = 1, \;
            q^4 = p^{2^{m - 3}}, \;
            q^{-1}pq = q^2p^{-1 + 2^{m - 3}}, \;
            q^{-2}pq^2 = p^{1 + 2^{m - 3}}
        \end{array}
        \Bigg\rangle,
    \]
    was omitted.

    \subsubsection{Class $\mathcal{N}_{II}(m)$}
    \label{subsec:ClassN_II}

    There are 9 non-isomorphic groups in class $\mathcal{N}_{II}(m)$ given as group presentations below~\cite{McKelden1906}. We denote these groups by $\mathcal{N}_{II, n}(m)$ for $1 \leq n \leq 9$. The 5-tuple of integral parameters $e_i$ defining each group is found in \textbf{Table \ref{tbl:NII,params}}.
    \[
        \mathcal{N}_{II, n}(m) =
        \Bigg\langle
        p, q, r \; \Bigg\mid
        \begin{array}{cc}
            p^{2^{m - 2}} = 1, \;
            q^2 = 1, \;
            r^2 = p^{2^{m - 3}e_5}, \;
            q^{-1}pq = p^{1 + 2^{m - 3}e_2}, \\
            r^{-1}pr = p^{e_1 + 2^{m - 3}e_3}, \;
            r^{- 1}qr = qp^{2^{m - 3}e_4}
        \end{array}
        \Bigg\rangle
    \]
    The completeness of the list of groups in this class was verified using~\cite{Ninomiya1994}.

    We concluded earlier that the groups belonging to class $\mathcal{N}_{II}(m)$ have rank at most 3. Since $|\langle{p, q}\rangle| = 2^{m - 1}$, to show that the rank of each group in this class is exactly 3, we just need to show that the orders of the subgroups $\langle{p, r}\rangle$ and $\langle{q, r}\rangle$ of $\mathcal{N}_{II, n}(m)$ is less than $2^m$. Based on the relations in the presentation of $\mathcal{N}_{II, n}(m)$, we conclude that
    \begin{itemize}
        \item every element of $\langle{p, r}\rangle$ can be written uniquely in the form $p^ir^k$, where $i \in \mathbf{Z}_{2^{m - 2}}$ and $k \in \mathbf{Z}_2$ and
        \item every element of $\langle{q, r}\rangle$ can be written uniquely in the form $p^iq^jr^k$, where $i \in \{0 \mod 2^{m - 2}, 2^{m - 3} \mod 2^{m - 2}\}$ and $j, k \in \mathbf{Z}_2$.
    \end{itemize}

    Consequently, $|\langle{p, r}\rangle|$, $|\langle{q, r}\rangle| < 2^m$. Thus, $\{p, q, r\}$ is a minimal generating set for $\mathcal{N}_{II, n}(m)$.

    \subsection{$\text{exp}(\mathcal{G}) = 2^{m - 3}$, $m \geq 7$}
    \label{subsec:2^m-3}

    Assume $m \geq 7$. Let $p$ be an element of $\mathcal{G}$ with $\text{ord}(p) = 2^{m - 3}$. Then $[\mathcal{G} : \langle{p}\rangle] = 8$ and $g^8 \in \langle{p}\rangle$ for any $g \in \mathcal{G}$. We consider three possibilities.
    \begin{itemize}
        \item Suppose there exists $q \in \mathcal{G} - \langle{p}\rangle$ such that $q^4 \notin \langle{p}\rangle$. Then $q^i\langle{p}\rangle$ for $0 \leq i \leq 7$ are the eight distinct cosets of $\langle{p}\rangle$ in $\mathcal{G}$. It follows immediately that $\mathcal{G} = \langle{p, q}\rangle$ and, therefore, $\text{rank}(\mathcal{G}) = 2$. We denote by $\mathcal{M}_{I}(m)$ the class to which these groups belong.
        \item Suppose that $g^4 \in \langle{p}\rangle$ for any $g \in \mathcal{G}$ and that there exists $q \in \mathcal{G} - \langle{p}\rangle$ such that $q^2 \notin \langle{p}\rangle$. Then $q^i\langle{p}\rangle$ for $0 \leq i \leq 3$ are four distinct cosets of $\langle{p}\rangle$ in $\mathcal{G}$. So the elements $p^iq^j$, where $0 \leq i \leq 2^{m - 3} - 1$ and $0 \leq j \leq 3$, make up half of the elements of $\mathcal{G}$, and therefore, $|\langle{p, q}\rangle| \geq 2^{m - 1}$. Two cases arise. If $|\langle{p, q}\rangle| = 2^{m - 1}$, then $\langle{p, q}\rangle$ must belong to class $\mathcal{N}_I(m - 1)$ and there exists $r \in \mathcal{G} - \langle{p, q}\rangle$ such that $\mathcal{G} = \langle{p, q, r}\rangle$. Therefore, $\text{rank}(\mathcal{G}) \leq 3$. We denote by $\mathcal{M}_{II}(m)$ the class to which these groups belong. If $|\langle{p, q}\rangle| = 2^m$, on the other hand, then $\mathcal{G} = \langle{p, q}\rangle$ and, therefore, $\text{rank}(\mathcal{G}) = 2$.  We denote by $\mathcal{M}_{II'}(m)$ the class to which these groups belong.
        \item Suppose that $g^2 \in \langle{p}\rangle$ for any $g \in \mathcal{G}$. Using an argument similar to that used to describe the groups in class $\mathcal{N}_{II}(m)$, we conclude that we can find $q \in \mathcal{G} - \langle{p}\rangle$ and $r \in \mathcal{G} - \langle{p, q}\rangle$ such that $|\langle{p, q, r}\rangle| = 2^{m - 1}$ and $\langle{p, q, r}\rangle$ belongs to class $\mathcal{N}_{II}(m - 1)$. Thus, there exists $s \in \mathcal{G} - \langle{p, q, r}\rangle$ such that $\mathcal{G} = \langle{p, q, r, s}\rangle$ and, therefore, $\text{rank}(\mathcal{G}) \leq 4$. We denote by $\mathcal{M}_{III}(m)$ the class to which these groups belong.
    \end{itemize}

    The presentations of groups belonging to the above three classes were constructed and tabulated by McKelden in~\cite{McKelden1906}. Each group is parameterized by integers obtained by solving a system of congruence equations. Each solution to the system corresponds to powers of the generators in the group relations of the presentation. McKelden also showed that other combinations of parameters will result to groups that are isomorphic to those already in the list. She demonstrated this by giving the corresponding group isomorphisms, which were expressed in terms of a separate series of congruence equations. However, there are some combination of parameters that have been left out. In addition, the list contains a few groups of order less than $2^m$. Further, some pairs of groups, which are actually isomorphic, have been listed down separately.

    For classes $\mathcal{M}_{II}(m)$ and $\mathcal{M}_{III}(m)$ containing groups of rank 3 and 4, respectively, we shall complete the list of groups; make parameter adjustments to the groups of order less than $2^m$; and identify those groups which are isomorphic. These tasks were accomplished with the help of GAP by observing the steps below. \\

    \noindent \textbf{Step 1.} For a given group presentation and a given value of $m$, run GAP to find all possible combinations of parameters that will give rise to non-isomorphic groups of order $2^m$ and exponent $2^{m - 3}$. \\

    \noindent \textbf{Step 2.} Show that those combinations of parameters that will yield isomorphic groups for the given value of $m$ in \textbf{Step 1} will indeed yield isomorphic groups for any $m \geq 7$ by explicitly giving the isomorphisms. Note that a group isomorphism may be generated using GAP for particular value of $m$. This isomorphism may then be generalized for any $m \geq 7$ by means of a substitution of generators. \\

    \subsubsection{Class $\mathcal{M}_{II}(m)$}
    \label{subsec:ClassM_II}

    Each of the 123 non-isomorphic groups in class $\mathcal{M}_{II}(m)$ is classified into seven subclasses according to the subclass of $\mathcal{N}_{I}(m - 1)$ to which its subgroup $\langle{p, q}\rangle$ belongs. The presentations of the groups from~\cite{McKelden1906} are reproduced below. The sets of integral parameters $e_i$ defining the groups with the corrections and additions already incorporated are found in the tables of \textbf{\ref{app:params}}. \\

    \[
        \mathcal{M}_{II-A, n}(m) =
        \Bigg\langle
        p, q, r \; \Bigg\mid
        \begin{array}{cc}
            p^{2^{m - 3}} = 1, \;
            q^4 = 1, \;
            r^2 = q^{2e_3}p^{2^{m - 4}e_4}, \;
            q^{-1}pq = p^{1 + 2^{m - 5}e_7}, \\
            r^{-1}pr = q^{2e_1}p^{e_8 + 2^{m - 4}e_5}, \;
            r^{- 1}qr = q^{1 + 2e_2}p^{2^{m - 5}e_6}
        \end{array}
        \Bigg\rangle,
    \]

    \[
        \mathcal{M}_{II-B, n}(m) =
        \Bigg\langle
        p, q, r \; \Bigg\mid
        \begin{array}{cc}
            p^{2^{m - 3}} = 1, \;
            q^4 = 1, \;
            r^2 = q^{2e_3}, \;
            q^{-1}pq = p^{-1 + 2^{m - 5}e_6}, \\
            r^{-1}pr = q^{2e_1}p^{e_7 + 2^{m - 4}e_5}, \;
            r^{-1}qr = q^{1 + 2e_2}p^{2^ue_4}
        \end{array}
        \Bigg\rangle,
    \]

    \[
        \mathcal{M}_{II-C, n}(m) =
        \Bigg\langle
        p, q, r \; \Bigg\mid
        \begin{array}{cc}
            p^{2^{m - 3}} = 1, \;
            q^4 = p^{2^{m - 4}}, \;
            r^2 = 1, \;
            q^{-1}pq = p^{-1}, \\
            r^{-1}pr = p^{1 + 2^{m - 4}e_2}, \;
            r^{- 1}qr = q^{1 + 2e_1}
        \end{array}
        \Bigg\rangle,
    \]

    \[
        \mathcal{M}_{II-D, n}(m) =
        \Bigg\langle
        p, q, r \; \Bigg\mid
        \begin{array}{cc}
            p^{2^{m - 3}} = 1, \;
            q^4 = 1, \;
            r^2 = q^{2e_3}, \;
            q^{-1}pq = q^2p^{-1 + 2^{m - 4}e_7}, \\
            q^{-2}pq^2 = p^{1 + 2^{m - 4}e_6}, \;
            r^{-1}pr = q^{2e_1}p^{e_8 + 2^{m - 4}e_5}, \;
            r^{-1}qr = q^{1 + 2e_2}p^{2^ue_4}
        \end{array}
        \Bigg\rangle,
    \]

    \[
        \mathcal{M}_{II-E, n}(m) =
        \Bigg\langle
        p, q, r \; \Bigg\mid
        \begin{array}{cc}
            p^{2^{m - 3}} = 1, \;
            q^4 = 1, \;
            r^2 = 1, \;
            q^{-1}pq = q^2p^{1 - 2^{m - 5}e_2}, \\
            q^{-2}pq^2 = p^{1 + 2^{m - 4}e_2}, \;
            r^{-1}pr = p^{e_1 + 2^{m - 4}e_4}, \;
            r^{-1}qr = qp^{2^{m - 5}e_3}
        \end{array}
        \Bigg\rangle,
    \]

    \[
        \mathcal{M}_{II-F, n}(m) =
        \Bigg\langle
        p, q, r \; \Bigg\mid
        \begin{array}{cc}
            p^{2^{m - 3}} = 1, \;
            q^4 = p^{2^{m - 4}}, \;
            r^2 = 1, \;
            q^{-1}pq = q^2p^{-1 + 2^{m - 4}e_5}, \\
            q^{-2}pq^2 = p^{1 + 2^{m - 4}e_1}, \;
            r^{-1}pr = q^{2e_3}p^{1 + 2^{m - 4}e_2}, \;
            r^{-1}qr = q^{1 + 2e_3}p^{2^{m - 4}e_4}
        \end{array}
        \Bigg\rangle,
    \]

    \[
        \mathcal{M}_{II-G, n}(m) =
        \Bigg\langle
        p, q, r \; \Bigg\mid
        \begin{array}{cc}
            p^{2^{m - 3}} = 1, \;
            q^4 = p^4, \;
            r^2 = 1, \;
            q^{-1}pq = q^2p^{-1 + 2^{m - 5}e_4}, \\
            q^{-2}pq^2 = p^{1 + 2^{m - 4}e_1}, \;
            r^{-1}pr = p^{1 + 2^{m - 4}e_3}, \;
            r^{-1}qr = qp^{2^{m - 4}e_2}
        \end{array}
        \Bigg\rangle,
    \]

    \noindent where
    \[
            u =
            \begin{cases}
                1 & \text{if $n = 27, 28$ for $\mathcal{M}_{II-B}(m)$ or if $n = 10, 11$ for $\mathcal{M}_{II-D}(m)$}, \\
                m - 4 & \text{otherwise}. \\
            \end{cases}
    \]

    As before, we can use each of the relations above to show that rank of each group in class $\mathcal{M}_{II}(m)$ is exactly 3. Since $|\langle{p, q}\rangle| = 2^{m - 1}$ by assumption, it suffices to show that $|\langle{p, r}\rangle|$, $|\langle{q, r}\rangle| < 2^m$. We omit the details for brevity.

    In~\cite{McKelden1906}, 46 groups were included in subclass $\mathcal{M}_{II-A}(m)$. However, the groups $\mathcal{M}_{II-A, 21}(m)$ and $\mathcal{M}_{II-A, 22}(m)$ are isomorphic with isomorphism given by $p_1 \mapsto p_2$, $q_1 \mapsto p_2^{2^{m - 5}}q_2$, $r_1 \mapsto p_2q_2r_2$. So are the groups $\mathcal{M}_{II-A, 23}(m)$ and $\mathcal{M}_{II-A, 24}(m)$ with isomorphism given by $p_1 \mapsto p_2$, $q_1 \mapsto q_2$, $r_1 \mapsto p_2q_2r_2$. This reduces the number of non-isomorphic groups in the said subclass to 44. In addition, for the group $\mathcal{M}_{II-A, 27}(m)$, we change the value of $e_7$ from the original 1 to 2. Note that if $e_7 = 1$, then $|\mathcal{M}_{II-A, 27}(m)| = 2^{m - 1}$.

    There are 28 non-isomorphic groups in subclass $\mathcal{M}_{II-B}(m)$ including the additional groups $\mathcal{M}_{II-B, n}(m)$ for $n = 12, 21, 22, 23$. These are found by running a search for 7-tuples $(e_1, e_2, e_3, e_4, e_5, e_6, e_7)$ of parameters $e_1, e_2, e_3, e_4, e_5 \in \mathbf{Z}_2$, $e_6 \in \mathbf{Z}_4$, and $e_7 \in \{-1, 1\}$ that will generate new groups. Lastly, for $\mathcal{M}_{II-B, 28}(m)$, McKelden imposed the restriction that $m > 7$. We may, however, also include the case when $m = 7$ for this group.

    There are 5, 9, and 7 non-isomorphic groups in subclass $\mathcal{M}_{II-C}(m)$, $\mathcal{M}_{II-F}(m)$, and $\mathcal{M}_{II-G}(m)$, respectively. These counts include the additional groups $\mathcal{M}_{II-C, 5}(m)$, $\mathcal{M}_{II-F, 3}(m)$, $\mathcal{M}_{II-G, 5}$, and $\mathcal{M}_{II-G, 7}$ belonging to their respective subclass. These are found by running a search for pairs $(e_1, e_2)$ of parameters $e_1 \in \mathbf{Z}_4$ and $e_2 \in \mathbf{Z}_2$ for $\mathcal{M}_{II-C}(m)$; 5-tuples $(e_1, e_2, e_3, e_4, e_5)$ of parameters with $e_1, e_2, e_4, e_5 \in \mathbf{Z}_2$ and $e_3 \in \mathbf{Z}_4$ for $\mathcal{M}_{II-F}(m)$; and 4-tuples $(e_1, e_2, e_3, e_4)$ of parameters with $e_1, e_2, e_3 \in \mathbf{Z}_2$ and $e_4 \in \mathbf{Z}_2$ for $\mathcal{M}_{II-G}(m)$ that will generate groups not yet in the list.

    The subclass $\mathcal{M}_{II-D}(m)$ contains 22 non-isomorphic groups including the additional group $\mathcal{M}_{II-D, 8}(m)$. These are found by running a search for 8-tuples $(e_1, e_2, e_3, e_4, e_5, e_6, e_7, e_8)$ of parameters  $e_1, e_2, e_3, e_4, e_5, e_6, e_7 \in \mathbf{Z}_2$ and $e_8 \in \{-1, 1\}$ that will generate new groups. The groups $\mathcal{M}_{II-D, 15}(m)$ and $\mathcal{M}_{II-D, 10}(m)$ were originally treated separately. However, they are isomorphic with isomorphism given by $p_1 \mapsto q_2r_2$, $q_1 \mapsto q_2$, $r_1 \mapsto q_2p_2r_2$.

    Finally, the subclass $\mathcal{M}_{II-E}(m)$ contains 8 non-isomorphic groups including the additional group $\mathcal{M}_{II-E, 8}(m)$.  This is found by running a search for 4-tuples $(e_1, e_2, e_3, e_4)$ of parameters $e_1 \in \{-1, 1\}$, $e_2, e_3 \in \mathbf{Z}_4$, and $e_4 \in \mathbf{Z}_2$ that will generate groups not yet in the list. For $\mathcal{M}_{II-E, 4}(m)$, we change the value of $(e_3, e_4)$ from the original $(1, 0)$ to $(2, 1)$. Likewise, for $\mathcal{M}_{II-E, 7}(m)$, we change the value of $e_3$ from the original $1$ to $2$. Note that given the original values of the parameters, $|\mathcal{M}_{II-E, 4}(m)| = |\mathcal{M}_{II-E, 7}(m)| = 2^{m - 1}$.


    \subsubsection{Class $\mathcal{M}_{III}(m)$}
    \label{subsec:ClassM_III}

    There are 14 non-isomorphic groups in class $\mathcal{M}_{III}(m)$ including the additional groups $\mathcal{M}_{III, n}(m)$ for $n = 2, 12, 13, 14$. The presentation of each of these groups is given below.
    \[
        \mathcal{M}_{III, n}(m) =
        \Biggl\langle
        p, q, r, s \; \Biggl\mid\Biggr.
        \begin{array}{cc}
            p^{2^{m - 3}} = 1, \;
            q^2 = 1, \;
            r^2 = p^{2^{m - 4}e_8}, \;
            s^2 = p^{2^{m - 4}e_5}, \;
            q^{-1}pq = p, \\
            r^{-1}pr = p^{e_2 + 2^{m - 4}e_3}, \;
            s^{-1}ps = p^{e_1 + 2^{m - 4}e_6}, \;
            r^{-1}qr = q, \\
            s^{-1}qs = qp^{2^{m - 4}e_4}, \;
            s^{-1}rs = rp^{2^{m - 4}e_7}
        \end{array}
        \Bigg\rangle
    \]

    We assert that each group in class $\mathcal{M}_{III}(m)$ has rank exactly 4. Since $|\langle{p, q, r}\rangle| = 2^{m - 1}$, this assertion can be verified by showing that the order of each of the subgroups $\langle{p, q, s}\rangle$, $\langle{p, r, s}\rangle$, $\langle{q, r, s}\rangle$ of $\mathcal{M}_{III, n}(m)$ is less than $2^m$. We omit the details for brevity.

    The additional groups in the class are found by running a search for 8-tuples $(e_1, e_2, e_3, e_4, e_5, e_6, e_7, e_8)$ of parameters $e_1, e_2 \in \{-1, 1\}$ and $e_3, e_4, e_5, e_6, e_7, e_8 \in \mathbf{Z}_2$ that will generate groups not yet in the list.

    The following result whose proof is straightforward by means of a substitution of generators gives the subgroup relations between the groups in class $\mathcal{M}_{III, n}(m) = \langle{p, q, r, s}\rangle$ and in class $\mathcal{N}_{II, n}(m - 1)$.
    \begin{proposition}
        The subgroup $\langle{p, r, s}\rangle$ or $\langle{p, q, r}\rangle$ of $\mathcal{M}_{III, n}(m)$ is isomorphic to $\mathcal{N}_{II, n'}(m) = \langle{p', q', r'}\rangle$ via the isomorphism
        \[
            \alpha =
            \begin{cases}
                p \mapsto p', \: r \mapsto q', \: s \mapsto r' & \text{ if } (n, n') \in \{(1, 1), (2, 2), \ldots, (9, 9)\}, \\
                p \mapsto p', \: q \mapsto q', \: r \mapsto r' & \text{ if } (n, n') \in \{(10, 5), (11, 6), (12, 8), (13, 8), (14, 7)\}. \\
            \end{cases}
        \]
        \label{prop:M_III(m),N_II(m-1)rels}
    \end{proposition}

    The next proposition gives the product decompositions of groups in class $\mathcal{M}_{III}(m)$.
    \begin{proposition}
        The group $\mathcal{M}_{III, n}(m)$ can be decomposed as a product of proper subgroups as follows.
        \[
            \langle{p, q, r, s}\rangle \simeq
            \begin{cases}
                \langle{p, r, s}\rangle \times \langle{q}\rangle & \text{ if } n = 1, 2, \dots, 9, \\
                \langle{p, q, r}\rangle\langle{s}\rangle & \text{ if } n = 10, 11, \\
                \langle{p, q, r}\rangle \rtimes \langle{s}\rangle & \text{ if } n = 12, 13, 14. \\
            \end{cases}
        \]
        \label{prop:M_III,productdecomposition}
    \end{proposition}

    \subsection{$\text{inv}(\mathcal{G})$}
    \label{subsec:invsG}

    Let $\mathcal{G}$ be a group in class $\mathcal{N}_{II}(m)$, $\mathcal{M}_{II}(m)$, or $\mathcal{M}_{III}(m)$. As can be observed from the presentation of $\mathcal{G}$, each element of the group can be expressed uniquely in the form $p^{\bar{i}}q^{\bar{j}}r^{\bar{k}}s^{\bar{l}}$ subject to the following restrictions on $\bar{i}, \bar{j}, \bar{k}, \bar{l}$:
    \[
        0 \leq \bar{i} < \text{ord}(p),
        \;\;\;\;\;
        0 \leq \bar{j} < \max_{0 \leq j' < \text{ord}(q)}\{j' \mid q^u \notin \langle{p}\rangle \text{ for all } u < j'\},
    \]
    \[
        0 \leq \bar{k} < \max_{0 \leq k' < \text{ord}(r)}\{k' \mid r^u \notin \langle{p, q}\rangle \text{ for all } u < k'\},
    \]
    \[
        0 \leq \bar{l} < \max_{0 \leq l' < \text{ord}(s)}\{l' \mid s^u \notin \langle{p, q, r}\rangle \text{ for all } u < l'\}
    \]
    with $\bar{l} = 0$ whenever $\text{rank}(\mathcal{G}) = 3$.

    To find the set $\text{inv}(\mathcal{G})$ of involutions of $\mathcal{G}$, we solve for $i \in \mathbf{Z}_{\text{ord}(p)}$, $j \in \mathbf{Z}_{\text{ord}(q)}$, $k \in \mathbf{Z}_{\text{ord}(r)}$, and $l \in \mathbf{Z}_{\text{ord}(s)}$ the equation $(p^iq^jr^ks^l)^2 = 1$, which can be reexpressed uniquely in the form $p^{\bar{i}}q^{\bar{j}}r^{\bar{k}}s^{\bar{l}} = 1$. Solving the latter equation is then equivalent to solving the system
    \[
        \begin{cases}
            \bar{i} \equiv 0 \pmod{\text{ord}(p)}, \;\;\;\; \bar{j} \equiv 0 \pmod{\text{ord}(q)} \\
            \bar{k} \equiv 0 \pmod{\text{ord}(r)}, \;\;\;\; \bar{l} \equiv 0 \pmod{\text{ord}(s)}
        \end{cases}.
    \]
    of congruence equations.

    From these solutions, we arrive at the set of involutions for each group in class $\mathcal{N}_{II}(m)$ in \textbf{Table \ref{tbl:NII,invslist}} and each group in classes $\mathcal{M}_{II}(m)$ and $\mathcal{M}_{III}(m)$ in \textbf{Table \ref{tbl:MII,MIII,invslist}}. In these tables, we classify each involution as \textit{central} (commuting with every element of the group) or \textit{non-central} (non-commuting with at least one element of the group) and provide the total count of all involutions.

    To show that most groups in these classes are not generated by involutions, we use \textbf{Corollary \ref{cor:notgenbyinvs}}. For each group $\mathcal{G}$ that cannot be generated by involutions, we provided in the last column of \textbf{Table \ref{tbl:NII,invslist}} or \textbf{Table \ref{tbl:MII,MIII,invslist}}, a homomorphism $\varphi_X$ defined in \textbf{Proposition \ref{prop:2-groupHom}} that satisfies $\varphi_X(\langle{\text{inv}(\mathcal{G})}\rangle) = \{1\}$. It is easy to see that each remaining group not provided with a homomorphism is generated by involutions by noting that the generators $p$, $q$, $r$, and $s$ (if the group belongs to class $\mathcal{M}_{III}(m)$) can actually be obtained as a product of involutions.

    Therefore, we have the following theorem, which classifies the groups generated by involutions from the classes $\mathcal{N}_{II}(m)$, $\mathcal{M}_{II}(m)$, and $\mathcal{M}_{III}(m)$.
    \begin{proposition}
        The set $\Gamma$ defined below contains the groups from classes $\mathcal{N}_{II}(m)$, $\mathcal{M}_{II}(m)$, and $\mathcal{M}_{III}(m)$ which are generated by involutions. In this set, three groups have exponent $2^{m - 2}$ for $m \geq 6$ (all have rank 3) and nineteen have exponent $2^{m - 3}$ for $m \geq 7$. Among those with exponent $2^{m - 3}$, twelve have rank 3 and seven have rank 4.

        \[
            \Gamma =
            \left\{
                \begin{tabular}{llll}
                    $\mathcal{N}_{II, 3}(m)$, & $\mathcal{N}_{II, 4}(m)$, & $\mathcal{N}_{II, 5}(m)$, & $\mathcal{M}_{II-A, 39}(m)$, \\
                    $\mathcal{M}_{II-A, 40}(m)$, & $\mathcal{M}_{II-B, 19}(m)$, & $\mathcal{M}_{II-B, 20}(m)$, & $\mathcal{M}_{II-C, 5}(m)$, \\
                    $\mathcal{M}_{II-D, 3}(m)$, & $\mathcal{M}_{II-D, 17}(m)$, & $\mathcal{M}_{II-D, 18}(m)$, & $\mathcal{M}_{II-D, 19}(m)$, \\
                    $\mathcal{M}_{II-E, 1}(m)$, & $\mathcal{M}_{II-F, 8}(m)$, & $\mathcal{M}_{II-F, 9}(m)$, & $\mathcal{M}_{III, 3}(m)$, \\
                    $\mathcal{M}_{III, 4}(m)$, & $\mathcal{M}_{III, 5}(m)$, & $\mathcal{M}_{III, 10}(m)$, & $\mathcal{M}_{III, 11}(m)$, \\
                    $\mathcal{M}_{III, 13}(m)$, & $\mathcal{M}_{III, 14}(m)$ & &
                \end{tabular}
                \right\}.
        \]
        \label{prop:groupsGeneratedByInvs}
    \end{proposition} 
    
    \section{Classification of string C-groups of order $2^m$ and exponent at least $2^{m - 3}$}
\label{sec:stringC}

    The prime objective of this section is to prove the following theorem, which classifies the string C-groups of order $2^m$ and exponent at least $2^{m - 3}$. The corresponding diagrams of these string C-groups are provided in \textbf{Figure~\ref{fig:classificationdiagrams}}.

    \begin{theorem}
        Let $\mathcal{G}$ be a string C-group of order $2^m$.
        \begin{enumerate}
            \item If $\text{exp}(\mathcal{G}) = 2^m$, then $m = 1$ and $\mathcal{G} \simeq \mathbf{Z}_2$.
            \item If $\text{exp}(\mathcal{G}) = 2^{m - 1}$ with $m \geq 2$, then $\mathcal{G} \simeq \mathbf{D}_{2^{m - 1}}$.
            \item If $\text{exp}(\mathcal{G}) = 2^{m - 2}$ with $m \geq 3$, then $\mathcal{G}$ is disconnected and $\mathcal{G} \simeq \mathbf{D}_{2^{m - 2}} \times \mathbf{Z}_2$.
            \item If $\mathcal{G}$ is a disconnected string C-group with $\text{exp}(\mathcal{G}) = 2^{m - 3}$ for $m \geq 4$, then $\mathcal{G} \simeq \mathbf{D}_{2^{m - 3}} \times \mathbf{Z}_2 \times \mathbf{Z}_2$.
            \item If $\mathcal{G}$ is a connected string C-group with $\text{exp}(\mathcal{G}) = 2^{m - 3}$, then
                \[
                    \renewcommand*{\arraystretch}{0.8}
                    \mathcal{G} \simeq
                    \Bigg\langle
                    s_0, s_1, s_2 \; \Bigg\mid
                    \begin{array}{cc}
                        s_0^2 = 1, \;
                        s_1^2 = 1, \;
                        s_2^2 = 1, \;
                        (s_0s_1)^4 = 1, \;
                        (s_0s_2)^2 = 1, \\
                        (s_1s_2)^{2^{m - 3}} = 1, \;
                        (s_0s_1s_2s_1)^2 \cdot (s_2s_1)^{2^{m - 4}e} = 1 \\
                    \end{array}
                    \Bigg\rangle,
                \]
                where $e = 0$ for $m = 5$ and $e = 0, 1$ for $m \geq 6$. Furthermore, $\mathcal{G}$ is a smooth quotient of the infinite string Coxeter group $W = [4, 2^{m - 3}]$ and is isomorphic to $ (\mathbf{D}_{2^{m - 4}} \rtimes \mathbf{D}_2) \rtimes \mathbf{Z}_2$ or to $(\mathbf{D}_{2^{m - 4}} \times \mathbf{D}_2) \rtimes \mathbf{Z}_2$ according as $e = 0 \text{ or } 1$, respectively.
        \end{enumerate}
        \label{thm:onlyStringCGroups}
    \end{theorem}

    \begin{figure}[H]
        \centering
        \includegraphics[scale = 0.35, keepaspectratio = true]{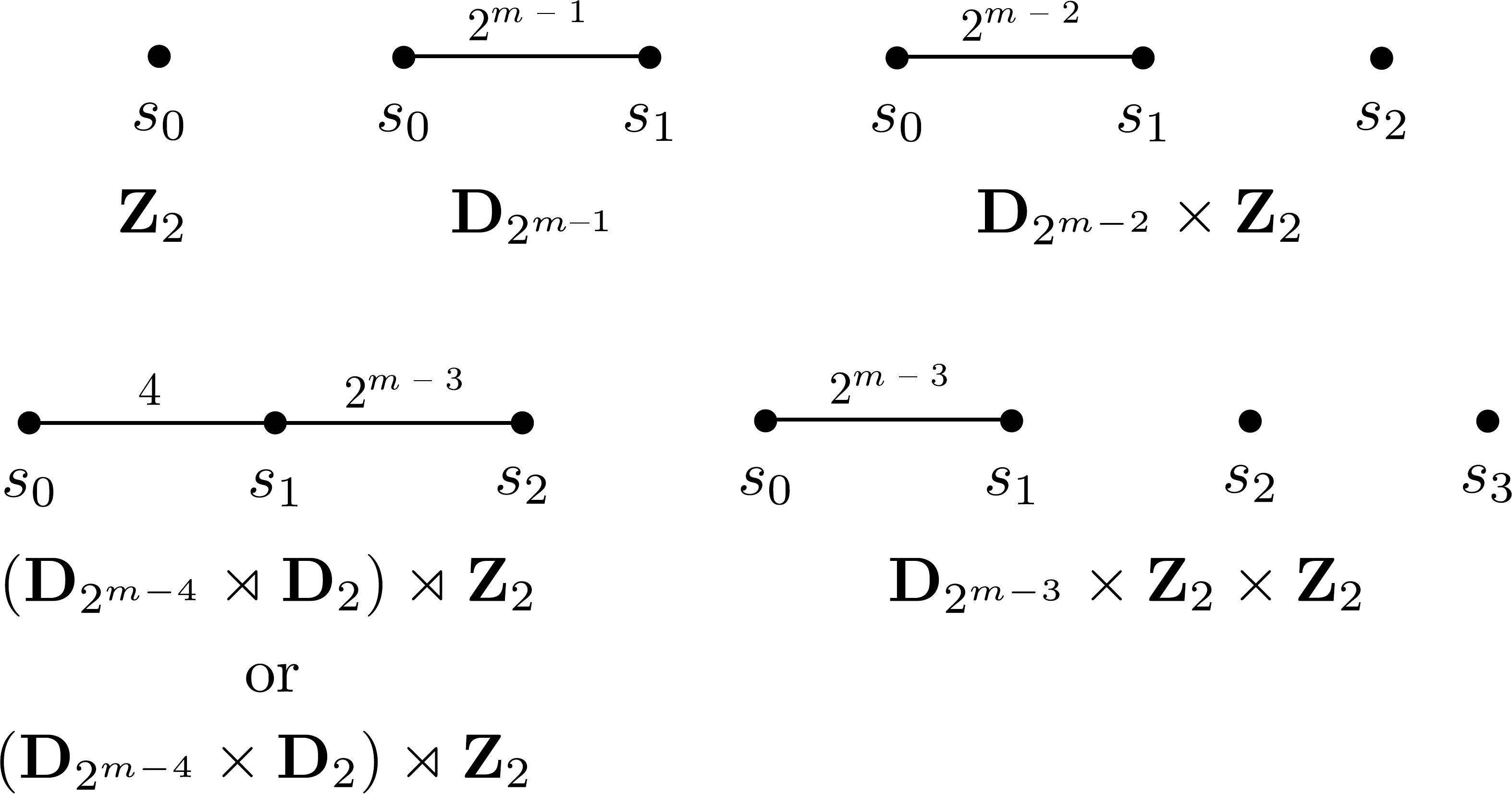}
        \caption{Diagrams of string C-groups in \textbf{Theorem \ref{thm:onlyStringCGroups}}.}
        \label{fig:classificationdiagrams}
    \end{figure}

    We first give a proof of the easier statements of \textbf{Theorem \ref{thm:onlyStringCGroups}}.
    \begin{proof}[Partial Proof of Theorem~\ref{thm:onlyStringCGroups}]
        Let $\mathcal{G}$ be a string C-group of order $2^m$.
        \begin{enumerate}
            \item \textit{If $\text{exp}(\mathcal{G}) = 2^m$, then $m = 1$ and $\mathcal{G} \simeq \mathbf{Z}_2$.} Let $\text{exp}(\mathcal{G}) = 2^m$, then $\mathcal{G}$ is a cyclic group isomorphic to $\mathbf{Z}_{2^m}$. Since the only cyclic group generated by an involution is $\mathbf{Z}_2$, we must have $m = 1$ and $\mathcal{G} \simeq \mathbf{Z}_2$.

            \item \textit{If $\text{exp}(\mathcal{G}) = 2^{m - 1}$ with $m \geq 2$, then $\mathcal{G} \simeq \mathbf{D}_{2^{m - 1}}$.} Let $\text{exp}(\mathcal{G}) = 2^{m - 1}$  with $m \geq 2$. Then $\mathcal{G}$ must have rank greater than 1. Let $p$ be an element of $\mathcal{G}$ with order $2^{m - 1}$. Then $\mathcal{G} = \langle{p}\rangle \cup q\langle{p}\rangle$ for some $q \notin \langle{p}\rangle$. It follows that $\mathcal{G}$ has rank 2. Since string C-groups of rank 2 are dihedral, we must have $\mathcal{G} \simeq \mathbf{D}_{2^{m - 1}}$.

            \item \textit{If $\mathcal{G}$ is a disconnected string C-group with $\text{exp}(\mathcal{G}) = 2^{m - 2}$ for $m \geq 3$, then $\mathcal{G} \simeq \mathbf{D}_{2^{m - 2}} \times \mathbf{Z}_2$.} Let $\text{exp}(\mathcal{G}) = 2^{m - 2}$ with $m \geq 3$ be disconnected. Then, from \textbf{Section \ref{sec:allgroups}}, we have $\text{rank}(\mathcal{G}) = 2 \text{ or } 3$. The only disconnected string C-group of rank 2 is isomorphic to $\mathbf{D}_2 \simeq \mathbf{Z}_2 \times \mathbf{Z}_2$, which has order $2^2$ and exponent $2^{2 - 1}$. Thus, $\text{rank}(\mathcal{G}) \neq 2$. On the other hand, the only disconnected string C-group of rank 3 is isomorphic to $\mathbf{D}_{p_{0, 1}} \times \mathbf{Z}_2$ with $p_{0, 1} \geq 2$. The diagram for this group is shown below.
                \begin{figure}[H]
                    \centering
                    \includegraphics[scale = 0.30, keepaspectratio = true]{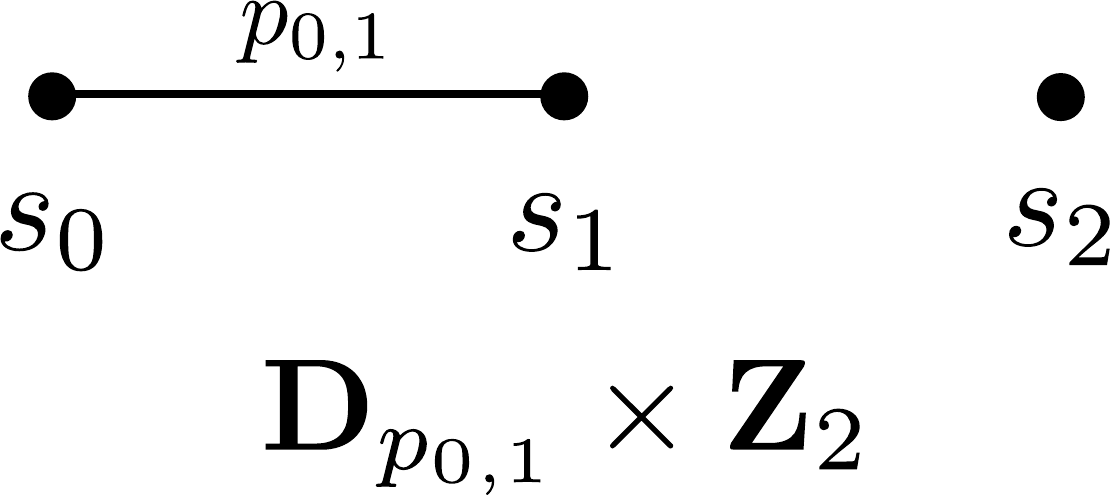}
                \end{figure}
                \noindent Note that the edge connecting the nodes corresponding to $s_0$ and $s_1$ must be deleted if $p_{0, 1} = 2$. If $\mathcal{G}$ were to have this diagram, we must have $p_{0, 1} = 2^{m - 2}$ so that $\text{ord}(\mathcal{G}) = 2^m$ and $\text{exp}(\mathcal{G}) \simeq 2^{m - 2}$. Hence, $\mathcal{G} \simeq \mathbf{D}_{2^{m - 2}} \times \mathbf{Z}_2$.

            \item \textit{If $\mathcal{G}$ is a disconnected string C-group with $\text{exp}(\mathcal{G}) = 2^{m - 3}$ for $m \geq 4$, then $\mathcal{G} \simeq \mathbf{D}_{2^{m - 3}} \times \mathbf{Z}_2 \times \mathbf{Z}_2$.} Let $\text{exp}(\mathcal{G}) = 2^{m - 3}$ with $m \geq 4$ be disconnected. Then, from \textbf{Section \ref{sec:allgroups}}, we have that $\text{rank}(\mathcal{G}) = 2, 3, \text{ or } 4$. An argument similar to the one used earlier in the case when $\text{exp}(\mathcal{G}) = 2^{m - 2}$ can be used in this case to show that $\text{rank}(\mathcal{G}) \neq 2$. Similarly, the only disconnected group of rank 3 is isomorphic to $\mathbf{D}_{p_{0, 1}} \times \mathbf{Z}_2$. If this group has exponent $2^{m - 3}$, then $p_{0, 1} = 2^{m - 3}$. However, this will result to a group of order $2^{m - 1}$ only. Thus, $\text{rank}(\mathcal{G}) \neq 3$. Now, a disconnected string C-group of rank 4 is isomorphic to either $\mathbf{D}_{p_{0, 1}} \times \mathbf{D}_{p_{2, 3}}$ or $\mathcal{H} \times \mathbf{Z}_2$, where $\mathcal{H}$ is a string C-group of order $2^{m - 1}$. The diagrams for these groups are shown below.
                \begin{figure}[H]
                    \centering
                    \includegraphics[scale = 0.30, keepaspectratio = true]{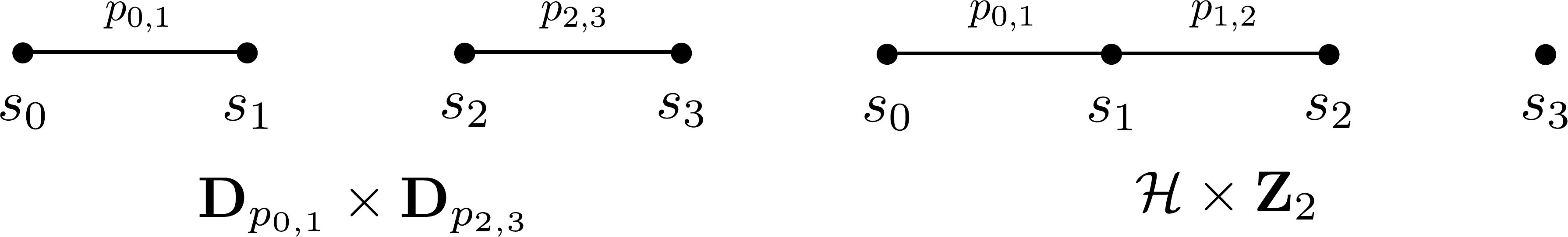}
                \end{figure}
                \noindent Note that the edge connecting the nodes corresponding to $s_j$ and $s_{j + 1}$ must be deleted if $p_{j, j + 1} = 2$. If $\mathcal{G}$ were to have the first diagram, we must have (up to duality) $p_{0, 1} = 2^{m - 3}$ and $p_{2, 3} = 2$ so that $\text{ord}(\mathcal{G}) = 2^m$ and $\text{exp}(\mathcal{G}) \simeq 2^{m - 3}$. If $\mathcal{G}$, on the other hand, were to have the second diagram, then $\mathcal{H}$ must have exponent $2^{m - 3}$. Thus, $\mathcal{H}$ is a group of order $2^{m'}$ and exponent $2^{{m'} - 2}$, where $m' = m - 1$. As will be shown in \textbf{Section~\ref{sec:stringC}}, $\mathcal{H}$ must be disconnected and, hence, from the previous result,  $ \mathcal{H} \simeq \mathbf{D}_{2^{m - 3}} \times \mathbf{Z}_2$. Thus, in either case, $\mathcal{G} \simeq \mathbf{D}_{2^{m - 3}} \times \mathbf{Z}_2 \times \mathbf{Z}_2$.
        \end{enumerate}
    \end{proof}

    To complete the proof of \textbf{Theorem \ref{thm:onlyStringCGroups}}, it remains to show that only the groups $\mathcal{M}_{II-D, 3}(m)$ and $\mathcal{M}_{II-D, 19}(m)$ in the set $\Gamma$ defined in the previous section are connected string C-groups and that both have type $\{4, 2^{m - 3}\}$. The proof combines the arguments in the proofs of two propositions found in the next two subsections. In the first subsection, we show that these two groups are indeed connected string C-groups of type $\{4, 2^{m - 3}\}$. In addition, we also give descriptions of their structures as products of proper subgroups (see \textbf{Proposition \ref{prop:M_II-D,3,19,StringCGroups}}). In the second subsection, we show that all the other remaining groups in the set $\Gamma$ (see \textbf{Proposition \ref{prop:groupsGeneratedByInvs}}) are not connected string C-groups (see \textbf{Proposition \ref{prop:Gamma-M_II-D,3,19,notstringC}}).

    \subsection{Completing the proof of \textbf{Theorem \ref{thm:onlyStringCGroups}}, Part 1}
    \label{subsec:M_II-D,3,19}

    Recall the group presentation \[
        \mathcal{M}_{II-D, n}(m) =
        \Bigg\langle
        p, q, r \; \Bigg\mid
        \begin{array}{cc}
            p^{2^{m - 3}} = 1, \;
            q^4 = 1, \;
            r^2 = 1, \;
            q^{-1}pq = q^2p^{-1}, \\
            q^{-2}pq^2 = p^{1 + 2^{m - 4}e_6}, \;
            r^{-1}pr = q^2p, \;
            r^{-1}qr = q^3
        \end{array}
        \Bigg\rangle,
    \]
    where $e_6 = 1$ if $n = 3$ and $e_6 = 0$ if $n = 19$, from \textbf{Section \ref{subsec:ClassM_II}}. Using the relations in the presentation, we can prove the following.
    \begin{proposition}
        $ $
        \begin{enumerate}
            \item The groups $\mathcal{M}_{II-D, 3}(m)$ and $\mathcal{M}_{II-D, 19}(m)$ are connected string C-groups of type $\{4, 2^{m - 3}\}$ with respect to the involutions $s_0 = r, s_1 = q^3r, s_2 = p^{-1}q^3r$.

            \item We have the product decompositions $\mathcal{M}_{II-D, 3}(m) \simeq (\mathbf{D}_{2^{m - 4}} \rtimes \mathbf{D}_2) \rtimes \mathbf{Z}_2$ and $\mathcal{M}_{II-D, 19}(m) \simeq (\mathbf{D}_{2^{m - 4}} \times \mathbf{D}_2) \rtimes \mathbf{Z}_2$.

            \item We have the alternative presentations
                \[
                    \mathcal{M}_{II-D, n}(m) \simeq
                    \Bigg\langle
                    s_0, s_1, s_2 \; \Bigg\mid
                    \begin{array}{cc}
                        s_0^2 = 1, \;
                        s_1^2 = 1, \;
                        s_2^2 = 1, \;
                        (s_0s_1)^4 = 1, \;
                        (s_0s_2)^2 = 1, \\
                        (s_1s_2)^{2^{m - 3}} = 1, \;
                        (s_0s_1s_2s_1)^2 \cdot (s_2s_1)^{2^{m - 4}e_6} = 1 \\
                    \end{array}
                    \Bigg\rangle,
                \]
                where $e_6 = 1$ if $n = 3$ and $e_6 = 0$ if $n = 19$. Thus, if $W = [4, 2^{m - 3}]$ is the Coxeter group with distinguished generators $w_0, w_1, w_2$, then
                \[ \mathcal{M}_{II-D, n}(m) \simeq W/\langle{(w_0w_1w_2w_1)^2 \cdot (w_2w_1)^{2^{m - 4}e_6}}\rangle^W. \]
            \end{enumerate}
            \label{prop:M_II-D,3,19,StringCGroups}
    \end{proposition}

    \begin{proof}
        Let $s_0 = r$, $s_1 = rq = q^3r$, and $s_2 = q^{-1}rp = p^{-1}q^3r$.
        \begin{enumerate}
            \item The presentations for $\mathcal{M}_{II-D, 3}(m)$ and $\mathcal{M}_{II-D, 19}(m)$ give us $s_0^{-1} = s_0$, $s_1^{-1} = s_1$, and $s_2^{-1} = s_2$. It is easy to see that the $s_i$'s generate $\mathcal{M}_{II-D, n}(m)$ by noting that $p = s_1s_2$, $q = s_0s_1$, and $r = s_0$. In addition, $\mathcal{M}_{II-D, n}(m)$ is a smooth quotient of $W = [4, 2^{m - 3}]$ since $\text{ord}(s_0s_2) = 2$, $\text{ord}(s_0s_1) = 4$, and $\text{ord}(s_1s_2) = 2^{m - 3}$. This shows that $\langle{s_0, s_1, s_2}\rangle$ has the following connected string diagram
                \begin{figure}[H]
                    \centering
                    \includegraphics[scale = 0.35, keepaspectratio = true]{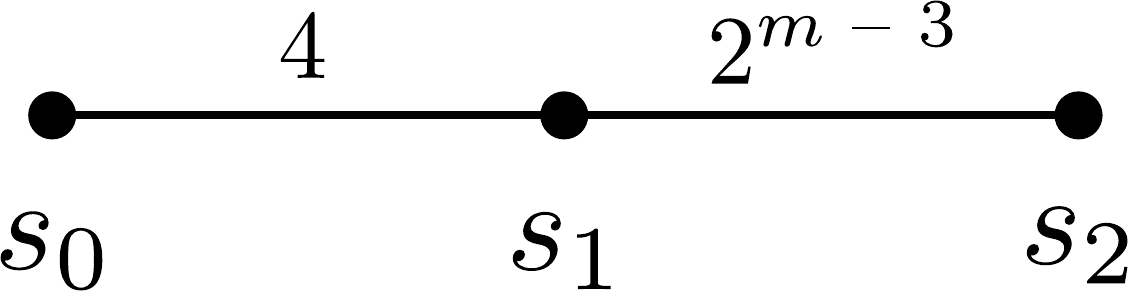}
                \end{figure}
                It remains to show that $\langle{s_0, s_1, s_2}\rangle$ satisfies the intersection condition. Set $H_{0, 1} := \langle{s_0, s_1}\rangle \cong \mathbf{D}_4$ and $H_{1, 2} := \langle{s_1, s_2}\rangle \cong \mathbf{D}_{2^{m - 3}}$. Then every element of $H_{0, 1}$ can be written in the form $(s_0s_1)^{x_1}{s_1}^{x_2} = q^{x_1}(rq)^{x_2}$ and every element of $H_{1, 2}$ can be written in the form $(s_1s_2)^{x_3}{s_1}^{x_4} = p^{x_3}(rq)^{x_4}$, where $0 \leq x_1 \leq 3$, $0 \leq x_3 \leq 2^{m - 3} - 1$, and $x_2, x_4 = 0, 1$. Suppose that for some $x_1, x_2, x_3, x_4$, we have $q^{x_1}(rq)^{x_2} = p^{x_3}(rq)^{x_4}$. If $x_2 - x_4 = 0$, then $q^{x_1} = p^{x_3}$. Since $q^2 \notin \langle{p}\rangle$ and $\text{ord}(q) = 4$, it follows that $x_1, x_3 = 0$. On the other hand, if $|x_2 - x_4| = 1$, then $q^{x_1}rq = q^{x_1}q^3r = q^{x_1 + 3}r = p^{x_3}$, which is equivalent to $p^{-x_3}q^{x_1 + 3} = r$. This cannot happen because $r \notin \langle{p, q}\rangle$. Thus, $H_{0, 1} \cap H_{1, 2} = \langle{s_1}\rangle$. Thus, $\langle{s_0, s_1, s_2}\rangle$ satisfies the intersection condition by \textbf{Lemma \ref{lem:rank3intcond}}.

            \item Let $H_0^{s_1} = \langle{s_0, s_1s_0s_1}\rangle \cong \mathbf{D}_2$ and $H_2^{s_1} = \langle{s_2, s_1s_2s_1}\rangle \cong \mathbf{D}_{2^{m - 4}}$. Then $H_0^{s_1}$ and $H_2^{s_1}$ are proper subgroups of $H_{0, 1} = \langle{s_0, s_1}\rangle$ and $H_{1, 2} = \langle{s_1, s_2}\rangle$, respectively, that do not contain $s_1$. It follows from the intersection condition that $H_0^{s_1} \cap H_2^{s_1} = \langle{e}\rangle$. In addition, $H_0^{s_1} \cdot H_2^{s_1} \cap \langle{s_1}\rangle = \langle{e}\rangle$, since $|H_0^{s_1} \cdot H_2^{s_1}| = 2^{m - 1}$.

                For $\mathcal{M}_{II-D, 3}(m)$, we claim that $\mathcal{M}_{II-D, 3}(m) \simeq  (H_2^{s_1} \rtimes H_0^{s_1}) \rtimes \langle{s_1}\rangle$. To prove this, it suffices to show that $H_2^{s_1}$ is invariant under conjugation by $s_0$. Direct computations yield $s_0s_2s_0 = s_2$ and $s_0(s_1s_2s_1)s_0 = s_1s_2s_1(s_2s_1s_2s_1)^{2^{m - 5}}$, which are clearly both in $H_2^{s_1}$.

                For $\mathcal{M}_{II-D, 19}(m)$, on the other hand, we claim that $\mathcal{M}_{II-D, 19}(m) \simeq (H_2^{s_1} \times H_0^{s_1}) \rtimes \langle{s_1}\rangle$. To prove this, it suffices to show that $s_0$ commutes with the generators of $H_2^{s_1}$. Computations yield $s_0s_2 = s_2s_0$ and $s_0(s_1s_2s_1) = (s_1s_2s_1)s_0$.

            \item The alternative presentations for $\mathcal{M}_{II-D, 3}(m)$ and $\mathcal{M}_{II-D, 19}(m)$ using the generators $s_0, s_1, s_2$ are obtained from the original presentations for these groups via the substitution $p := s_1s_2$, $q := s_0s_1$, and $r := s_0$.
        \end{enumerate}
    \end{proof}

    \subsection{Completing the proof of \textbf{Theorem \ref{thm:onlyStringCGroups}}, Part 2}
    \label{subsec:Gamma-M_II-D,3,19}

    We now show that except for $\mathcal{M}_{II-D, 3}(m)$ and $\mathcal{M}_{II-D, 19}(m)$ no other group in $\Gamma$ is a connected string C-group. To do so, we shall appeal to the following two lemmas, which give conditions under which the intersection condition is violated.
    \begin{lemma}
        Let $\mathcal{G} \in \Gamma$ and $s_0, s_1, s_2 \in \mathcal{G}$ be involutions with a connected string diagram.
        \begin{enumerate}
            \item If $(s_0s_1)^{2x_1} = (s_1s_2)^{2x_2} = \zeta$ is solvable in the variables $x_1, x_2$ for some central involution $\zeta \in \mathcal{G}$, then $\mathcal{G}$ is not a connected string C-group with respect to any set of distinguished generators containing $s_0, s_1, s_2$.
            \item If $s_1^{s_0} = s_1^{s_2}$, then $\mathcal{G}$ is not a connected string C-group with respect to any set of distinguished generators containing $s_0, s_1, s_2$.
        \end{enumerate}
        \label{lem:provingNotStringC1}
    \end{lemma}

    \begin{lemma}
        Let $\mathcal{G} \in \Gamma$. If for a fixed central involution $\zeta \in \mathcal{G}$, $(t_0t_1)^{2x} = \zeta$ is solvable in the variable $x$ for any pair of non-commuting involutions $t_0, t_1$ , then $\mathcal{G}$ is not a connected string C-group (with respect to any set of distinguished generators).
        \label{lem:provingNotStringC2}
    \end{lemma}

    The proof of \textbf{Lemma \ref{lem:provingNotStringC1}} hinges on the following implication: if $(s_0s_1)^{2x_1} = (s_1s_2)^{2x_2} = \zeta$ is solvable or if $s_1^{s_0} = s_1^{s_2}$, then $\langle{s_1}\rangle \lneqq \langle{s_1, \zeta}\rangle \leq \langle{s_0, s_1}\rangle \cap \langle{s_1, s_2}\rangle$. This follows from connectedness of the string diagram (that is, $s_1$ is not a central involution), which further implies that $s_1 \neq \zeta$. The hypothesis of \textbf{Lemma \ref{lem:provingNotStringC2}} is a more restrictive version of that of \textbf{Lemma \ref{lem:provingNotStringC1}}. Consequently, the proof of the former follows from that of the latter.

    It is important to mention that if the hypothesis of \textbf{Lemma \ref{lem:provingNotStringC1}} (resp. \textbf{Lemma \ref{lem:provingNotStringC2}}) is satisfied by the involutions $s_0, s_1, s_2$ (resp. $t_0, t_1$), then it is also satisfied by the involutions  $s_0' \in s_0Z(\mathcal{G})$, $s_1' \in s_1Z(\mathcal{G})$, $s_2' \in s_2Z(\mathcal{G})$ (resp. $t_0' \in t_0Z(\mathcal{G})$, $t_1' \in t_1Z(\mathcal{G})$). Thus, when using either lemma, it is enough to consider involutory coset representatives of $Z(\mathcal{G})$.

    The next lemma whose proof follows from the basic properties of cyclic groups, will be used as a companion result to \textbf{Lemma \ref{lem:provingNotStringC2}}.

    \begin{lemma}
        Let $\mathcal{G}$ be a 2-group. For any non-identity element $g \in \mathcal{G}$ and integer $c \nequiv 0 \pmod{\text{ord}(g)}$, the equation $g^{cx} = g^{\frac{\text{ord}(g)}{2}}$ has a solution in the variable $x$.
        \label{lem:gexsoln}
    \end{lemma}

    With this series of lemmas, we are now ready to prove the following proposition.

    \begin{proposition}
        Except for $\mathcal{M}_{II-D, 3}(m)$ and $\mathcal{M}_{II-D, 19}(m)$, none of the groups in $\Gamma$ is a connected string C-group. In particular,
        \begin{enumerate}
            \item each of the groups $\mathcal{N}_{II, 3}(m)$, $\mathcal{N}_{II, 4}(m)$, and $\mathcal{N}_{II, 5}(m)$ has the property that if $t_0$ and $t_1$ are non-commuting involutions, then the equation $(t_0t_1)^{2x} = p^{2^{m - 3}}$ is solvable in $x$;
            \item each of the groups $\mathcal{M}_{II-A, 39}(m)$, $\mathcal{M}_{II-A, 40}(m)$, $\mathcal{M}_{II-B, 19}(m)$, $\mathcal{M}_{II-B, 20}(m)$, $\mathcal{M}_{II-D, 17}(m)$, and $\mathcal{M}_{II-D, 18}(m)$ has the property that if $s_0, s_1, s_2$ are involutions with a connected string diagram, then $(s_0s_1)^{2x_1} = (s_1s_2)^{2x_2} = \zeta$ is solvable in $x_1, x_2$ for some central involution $\zeta$;
            \item each of the groups $\mathcal{M}_{II-C, 5}(m)$, $\mathcal{M}_{II-F, 8}(m)$, $\mathcal{M}_{II-F, 9}(m)$, $\mathcal{M}_{III, 3}(m)$, $\mathcal{M}_{III, 4}(m)$, $\mathcal{M}_{III, 5}(m)$, $\mathcal{M}_{III, 10}(m)$, $\mathcal{M}_{III, 11}(m)$, $\mathcal{M}_{III, 13}(m)$, and $\mathcal{M}_{III, 14}(m)$ has the property that if $t_0$ and $t_1$ are non-commuting involutions, then the equation $(t_0t_1)^{2x} = p^{2^{m - 4}}$ is solvable in $x$; and
            \item the group $\mathcal{M}_{II-E, 1}(m)$ has the property that if $s_0, s_1, s_2$ are involutions with a connected string diagram, then either $(s_0s_1)^{2x_1} = (s_1s_2)^{2x_2} = p^{2^{m - 4}}$ is solvable in $x_1, x_2$, or $s_1^{s_0} = s_1^{s_2}$.
        \end{enumerate}
        \label{prop:Gamma-M_II-D,3,19,notstringC}
    \end{proposition}

    \begin{proof}
        The proof is computational. For brevity, we show only the computations for some groups. The rest can be done in a similar fashion. In \textbf{Tables \ref{tbl:N_3,4,5,noncomminvs},  \ref{tbl:M_II-A,39,40,II-B,19,20,II-C,5,II-D,17,18,II-E,1,II-F,8,II-F,9,M_III,10,11,13,14,noncomminvs}}, we give the pairs of non-commuting involutory coset representatives $t_0$, $t_1$ of each of the centers of the groups in $\Gamma - \{\mathcal{M}_{II-D, 3}(m), \mathcal{M}_{II-D, 19}(m)\}$.

        \begin{enumerate}
            \item For $\mathcal{N}_{II, 3}(m)$, we have the pairs $(t_0, t_1) = (q, p^ir)$ and $(p^ir, p^{i'}r)$ , where $0 \leq i, i' \leq 2^{m - 2} - 1$ and $i' \neq i \mod{2^{m - 3}}$. We obtain the equations $(q \cdot p^ir)^{2x_1} = p^{2^{m - 3}x_1} = p^{2^{m - 3}}$ and $(p^ir \cdot p^{i'}r)^{2x_2} = p^{2(i - i')x_2} = p^{2^{m - 3}}$. By \textbf{Lemma \ref{lem:gexsoln}}, these equations are solvable.
            \item For $\mathcal{M}_{II-D, 17}(m)$, we have the pairs $(t_0, t_1) = (r, p^jq)$, $(r, p^kqr)$, $(p^jq, p^{j'}q)$, $(p^jq, p^kqr)$, and $(p^kqr, p^{k'}qr)$, where $j, j' = 1, 3, \ldots, 2^{m - 3} - 1$, $j' \neq j \mod{2^{m - 4}}$, $k, k' = 0, 2, \ldots, 2^{m - 3} - 2$, and $k' \neq k \mod{2^{m - 4}}$. We obtain the equations $(r \cdot p^jq )^{2x_1} = p^{2^{m - 4}x_1} = p^{2^{m - 4}}$, $(r \cdot p^kqr)^{2x_2} = q^{2x_2} = q^2$, $(p^jq \cdot p^{j'}q)^{2x_3} = p^{2(j - j')x_3} = p^{2^{m - 4}}$, $(p^jq \cdot p^kqr)^{2x_4} = p^{2(j - k)x_4} = p^{2^{m - 4}}$, and $(p^kqr \cdot p^{k'}qr)^{2x_5} = p^{2(k - k')x_5} = p^{2^{m - 4}}$. By \textbf{Lemma \ref{lem:gexsoln}}, these equations are solvable. Now let $s_0, s_1, s_2$ be involutory coset representatives of $Z(\mathcal{M}_{II-D, 17}(m))$ with a connected string diagram. Suppose that $(s_0s_1)^{2x_1} = (s_1s_2)^{2x_2} = p^{2^{m - 4}}$ has no solution in $x_1, x_2$. After switching $s_0$ and $s_2$, if needed, we may assume that $(s_0s_1)^{2x_1} = q^2$ for some $x_1$. From the above computations and the commutativity restriction on $s_0, s_2$, we have either $s_0, s_2 = r$ and $s_1 = p^kqr$, or $s_0, s_2 = p^kqr$ and $s_1 = r$. In any case, $(s_1s_2)^{2x_2} = q^2$ for some $x_2$. Thus, either $(s_0s_1)^{2x_1} = (s_1s_2)^{2x_2} = p^{2^{m - 4}}$ or $(s_0s_1)^{2x_1} = (s_1s_2)^{2x_2} = q^2$ is solvable.
            \item Let $n = 3, 4, 5$. Then, by \textbf{Proposition \ref{prop:M_III,productdecomposition}}, $\mathcal{M}_{III, n}(m)$ can be decomposed as $\langle{p, q, r, s}\rangle \simeq \langle{p, r, s}\rangle \times \langle{q}\rangle$, where $\langle{p, r, s}\rangle$ is isomorphic to $\mathcal{N}_{II, n}(m - 1)$. Thus, an involution $w \in \mathcal{M}_{III, n}(m)$ can be written in the form $w = uv$, where $u = 1$ or is an involution in $\langle{p, r, s}\rangle$ and $v = 1 \text{ or } q$. It follows from (1) after setting $m := m - 1$ that if $t_0$ and $t_1$ are non-commuting involutions in $\mathcal{M}_{III, n}(m)$, then $(t_0t_1)^{2x} = p^{2^{m - 4}}$ is solvable.
           \item Consider now the group $\mathcal{M}_{II-E, 1}(m)$. We look at those involutory coset representatives $s_0, s_1, s_2$ of $Z(\mathcal{M}_{II-E, 1}(m))$ with a connected string diagram for which $(s_0s_1)^{2x_1} = (s_1s_2)^{2x_2} = p^{2^{m - 4}}$ is not solvable. These triples (up to duality or re-indexing) are listed in \textbf{Table \ref{tbl:M_II-E,1,shortlist}}. A straightforward computation for each of these triples show that $s_1^{s_0} = s_1^{s_2}$.
               \begin{table}[H]
                   \footnotesize
                   \centering
                   \begin{tabu}{|l|l|l|}
                       \hline
                       \rowfont[r]{} $s_0$ & $s_1$ & $s_2$\\
                       \hline
                       \hline
                       $p^jr$ & $p^jqr$, $p^jq^3r$ & $p^jr$, $p^jq^2r$ \\
                       \hline
                       $p^jqr$ & $p^jr$, $p^jq^2r$ & $p^jqr$, $p^jq^3r$ \\
                       \hline
                       $p^jq^2r$ & $p^jqr$, $p^jq^3r$ & $p^jq^2r$ \\
                       \hline
                       $p^jq^3r$ & $p^jr$, $p^jq^2r$ & $p^jq^3r$ \\
                       \hline
                    \end{tabu}
                    \caption{Triples of involutory coset representatives $s_0$, $s_1$, $s_2$ of $Z(\mathcal{M}_{II-E, 1}(m))$ (up to duality or re-indexing) with a connected string diagram for which $(s_0s_1)^{2x_1} = (s_1s_2)^{2x_2} = p^{2^{m - 4}}$ is not solvable in $x_1, x_2$, where $j = 1, 3, \ldots, 2^{m - 3} - 1$.}
                    \label{tbl:M_II-E,1,shortlist}
                \end{table}
        \end{enumerate}
    \end{proof} 
    
    \section{Conclusion}
\label{sec:conc}

    This work provided a classification theorem for string C-groups of order $2^m$ and exponent at least $2^{m - 3}$. We have shown that except for the trivial cases (the cyclic group $\mathbf{Z}_2$ and the dihedral group $\mathbf{D}_{2^{m - 1}}$) only two of these groups are connected string C-groups. These are the groups $\mathcal{M}_{II-D,3}(m)$ and $\mathcal{M}_{II-D,19}(m)$. Both groups belong to class $\mathcal{M}_{II}(m)$, which means that they have rank 3 and exponent $2^{m - 3}$. In addition, both groups are smooth quotients of the infinite string Coxeter group $W = [4, 2^{m - 3}]$ and may be defined by the group presentations
         \begin{align*}
            \mathcal{M}_{II-D, 3}(m) &=
            \Bigg\langle
            s_0, s_1, s_2 \; \Bigg\mid
            \begin{array}{cc}
                s_0^2 = 1, \;
                s_1^2 = 1, \;
                s_2^2 = 1, \;
                (s_0s_1)^4 = 1, \;
                (s_0s_2)^2 = 1, \\
                (s_1s_2)^{2^{m - 3}} = 1, \;
                (s_0s_1s_2s_1)^2 \cdot (s_2s_1)^{2^{m - 4}} = 1 \\
            \end{array}
            \Bigg\rangle \\
            & \simeq (\mathbf{D}_{2^{m - 4}} \rtimes \mathbf{D}_2) \rtimes \mathbf{Z}_2, \\
            \mathcal{M}_{II-D, 19}(m) &=
            \Bigg\langle
            s_0, s_1, s_2 \; \Bigg\mid
            \begin{array}{cc}
                s_0^2 = 1, \;
                s_1^2 = 1, \;
                s_2^2 = 1, \;
                (s_0s_1)^4 = 1, \;
                (s_0s_2)^2 = 1, \\
                (s_1s_2)^{2^{m - 3}} = 1, \;
                (s_0s_1s_2s_1)^2 = 1 \\
            \end{array}
            \Bigg\rangle \\
            &\simeq (\mathbf{D}_{2^{m - 4}} \times \mathbf{D}_2) \rtimes \mathbf{Z}_2.
         \end{align*}
         Two groups in the classification are disconnected. One is isomorphic to $\mathbf{D}_{2^{m - 2}} \times \mathbf{Z}_2$  and the other is isomorphic to $\mathbf{D}_{2^{m - 3}} \times \mathbf{Z}_2 \times \mathbf{Z}_2$. The former belongs to class $\mathcal{N}_{II}(m)$ with rank 3 and exponent $2^{m - 2}$ while the latter belongs to class $\mathcal{M}_{III}(m)$ with rank 4 and exponent $2^{m - 3}$.

         The method of proof of the said classification theorem relied on the completeness of the list of the groups of order $2^m$, exponent at least $2^{m - 3}$, and rank at least 3. Previous literature provided the complete list of these groups except for those whose exponent is exactly $2^{m - 3}$. Thus, the task of completing the missing groups was accomplished prior to the classification. Using the presentations of the groups provided in McKelden's paper~\cite{McKelden1906}, we performed via GAP a parameter search that gave rise to groups that are not isomorphic to any of those already in the list. The search gave us a total of ten additional groups of rank 3 and four additional groups of rank 4 that were previously missing from McKelden's list.

         From the original 146 groups originally considered, we have narrowed down the list of candidates for string C-groups to 22 groups generated by involutions. This significant reduction was due to the use of a result, which we established, that gave a necessary condition for a 2-group to be generated by involutions. All the groups in the shortlist, except for $\mathcal{M}_{II-D,3}(m)$ and $\mathcal{M}_{II-D,19}(m)$, were shown not to be connected string C-groups by applying a series of lemmas whose respective hypotheses violate the intersection condition.

         The restriction on the exponent of a group was an important consideration in our classification. The need for this restriction is a testament to the difficulty of the general problem of characterizing string C-groups of order $2^m$ whose solution we primarily intended to contribute. In theory, the method presented in this work may be applied to classify string C-groups of order $2^m$ and exponent $2^{m - 4}$. As an initial step, one might consult~\cite{Babb1910, Finkel1906} for a list of known groups of this given exponent. It would certainly be interesting to investigate the deeper role of exponent in the characterization of string C-groups.
         
    \pagebreak
    \appendix
    \section{Set of parameters defining $\mathcal{G}$}
\label{app:params}

    \begin{table}[H]
        \footnotesize
        \centering
        \begin{tabular}{|c|ccccc|}
            \hline
            $n$ & $e_1$ & $e_2$ & $e_3$ & $e_4$ & $e_5$ \\
            \hline
            \hline
            $1$ & 1 & 0 & 0 & 1 & 0 \\
            \hline
            $2$ & $-1$ & 1 & 1 & 1 & 0 \\
            \hline
            $3$ & $-1$ & 0 & 0 & 1 & 0 \\
            \hline
            $4$ & $-1$ & 1 & 1 & 0 & 0 \\
            \hline
            $5$ & $-1$ & 0 & 0 & 0 & 0 \\
            \hline
            $6$ & $-1$ & 0 & 1 & 0 & 0 \\
            \hline
            $7$ & $-1$ & 0 & 0 & 0 & 1 \\
            \hline
            $8$ & 1 & 0 & 1 & 0 & 0 \\
            \hline
            $9$ & 1 & 0 & 0 & 0 & 0 \\
            \hline
        \end{tabular}
        \caption{Parameters defining $\mathcal{N}_{II, n}(m)$.}
        \label{tbl:NII,params}
    \end{table}

    \begin{table}[H]
        \footnotesize
        \centering
        \begin{tabular}{|c|cccccccc|c|cccccccc|}
            \hline
            $n$ & $e_1$ & $e_2$ & $e_3$ & $e_4$ & $e_5$ & $e_6$ & $e_7$ & $e_8$ &
            $n$ & $e_1$ & $e_2$ & $e_3$ & $e_4$ & $e_5$ & $e_6$ & $e_7$ & $e_8$ \\
            \hline
            \hline
            $1$ & 0 & 0 & 0 & 0 & 0 & 0 & 0 & 1 &
            $24^{**}$ & 1 & 1 & 1 & 0 & 1 & 2 & 2 & $-1$ \\
            \hline
            $2$ & 0 & 0 & 0 & 0 & 0 & 0 & 2 & 1 &
            $25$ & 0 & 1 & 0 & 1 & 0 & 0 & 0 & $-1$ \\
            \hline
            $3$ & 0 & 0 & 0 & 0 & 0 & 0 & 1 & 1 &
            $26$ & 0 & 1 & 1 & 0 & 0 & 0 & 0 & $-1$ \\
            \hline
            $4$ & 1 & 0 & 0 & 0 & 0 & 0 & 1 & 1 &
            $27$ & 0 & 1 & 1 & 0 & 0 & 0 & 2 & $-1$ \\
            \hline
            $5$ & 0 & 0 & 0 & 0 & 0 & 2 & 1 & 1 &
            $28$ & 0 & 1 & 1 & 0 & 1 & 0 & 0 & $-1$ \\
            \hline
            $6$ & 1 & 0 & 0 & 0 & 0 & 0 & 0 & 1 &
            $29$ & 0 & 1 & 1 & 0 & 0 & 2 & 0 & $-1$ \\
            \hline
            $7$ & 1 & 0 & 0 & 0 & 0 & 0 & 2 & 1 &
            $30$ & 0 & 0 & 0 & 0 & 0 & 0 & 0 & $-1$ \\
            \hline
            $8$ & 0 & 0 & 0 & 0 & 1 & 0 & 0 & 1 &
            $31$ & 0 & 0 & 0 & 0 & 0 & 0 & 2 & $-1$ \\
            \hline
            $9$ & 0 & 0 & 0 & 0 & 0 & 2 & 0 & 1 &
            $32$ & 0 & 1 & 0 & 0 & 0 & 2 & 0 & $-1$ \\
            \hline
            $10$ & 1 & 0 & 0 & 0 & 0 & 2 & 0 & 1 &
            $33$ & 1 & 1 & 0 & 0 & 0 & 0 & 0 & $-1$ \\
            \hline
            $11$ & 1 & 0 & 0 & 0 & 0 & 2 & 1 & 1 &
            $34$ & 1 & 1 & 0 & 0 & 0 & 0 & 2 & $-1$ \\
            \hline
            $12$ & 0 & 0 & 0 & 0 & 1 & 0 & 0 & $-1$ &
            $35$ & 0 & 1 & 0 & 0 & 1 & 0 & 0 & $-1$ \\
            \hline
            $13$ & 0 & 0 & 0 & 0 & 1 & 0 & 2 & $-1$ &
            $36$ & 0 & 1 & 0 & 0 & 1 & 0 & 0 & 1 \\
            \hline
            $14$ & 0 & 1 & 0 & 0 & 1 & 2 & 2 & $-1$ &
            $37$ & 0 & 1 & 0 & 0 & 0 & 0 & 0 & 1 \\
            \hline
            $15$  & 1 & 1 & 0 & 0 & 0 & 2 & 0 & $-1$ &
            $38$ & 0 & 1 & 0 & 0 & 0 & 0 & 2 & 1 \\
            \hline
            $16$ & 1 & 1 & 0 & 0 & 0 & 2 & 2 & $-1$ &
            $39$ & 0 & 1 & 0 & 0 & 0 & 0 & 0 & $-1$ \\
            \hline
            $17$ & 0 & 1 & 1 & 0 & 0 & 0 & 0 & 1 &
            $40$ & 0 & 1 & 0 & 0 & 0 & 0 & 2 & $-1$ \\
            \hline
            $18$ & 0 & 1 & 1 & 0 & 1 & 0 & 0 & 1 &
            $41$ & 0 & 0 & 0 & 0 & 0 & 1 & 0 & $-1$ \\
            \hline
            $19$ & 0 & 0 & 0 & 1 & 0 & 0 & 0 & $-1$ &
            $42$ & 0 & 0 & 0 & 0 & 0 & 1 & 2 & $-1$ \\
            \hline
            $20$ & 0 & 0 & 0 & 1 & 0 & 0 & 2 & $-1$ &
            $43$ & 0 & 0 & 0 & 0 & 0 & 0 & 1 & $-1$ \\
            \hline
            $21^*$ & 1 & 1 & 0 & 1 & 0 & 0 & 0 & $-1$ &
            $44$ & 0 & 0 & 0 & 0 & 1 & 1 & 1 & $-1$ \\
            \hline
            $22^*$ & 1 & 1 & 1 & 0 & 1 & 0 & 0 & $-1$ &
            $45$ & 0 & 1 & 0 & 0 & 0 & 1 & 0 & 1 \\
            \hline
            $23^{**}$ & 1 & 1 & 0 & 1 & 0 & 0 & 2 & $-1$ &
            $46$ & 0 & 1 & 0 & 0 & 0 & 1 & 2 & 1 \\
            \hline
        \end{tabular}
        \caption{Parameters defining $\mathcal{M}_{II-A, n}(m)$. $\mathcal{M}_{II-A, 21}(m) \simeq \mathcal{M}_{II-A, 22}(m)$ and $\mathcal{M}_{II-A, 23}(m) \simeq \mathcal{M}_{II-A, 24}(m)$.}
        \label{tbl:MII-A,params}
    \end{table}

    \begin{table}[H]
        \footnotesize
        \centering
        \begin{tabular}{|c|ccccccc|c|ccccccc|}
            \hline
            $n$ & $e_1$ & $e_2$ & $e_3$ & $e_4$ & $e_5$ & $e_6$ & $e_7$ &
            $n$ & $e_1$ & $e_2$ & $e_3$ & $e_4$ & $e_5$ & $e_6$ & $e_7$ \\
            \hline
            \hline
            $1$ & 0 & 0 & 0 & 0 & 0 & 0 & 1 &
            $15$ & 1 & 0 & 0 & 0 & 1 & 2 & 1 \\
            \hline
            $2$ & 0 & 0 & 0 & 0 & 0 & 2 & 1 &
            $16$ & 1 & 0 & 1 & 1 & 0 & 0 & 1 \\
            \hline
            $3$ & 0 & 0 & 0 & 0 & 0 & 1 & 1 &
            $17$ & 1 & 0 & 1 & 1 & 1 & 0 & 1 \\
            \hline
            $4$ & 1 & 1 & 0 & 1 & 0 & 1 & 1 &
            $18$ & 1 & 0 & 1 & 1 & 1 & 2 & 1 \\
            \hline
            $5$ & 0 & 0 & 0 & 1 & 0 & 1 & 1 &
            $19$ & 0 & 1 & 0 & 0 & 0 & 0 & 1 \\
            \hline
            $6$ & 1 & 0 & 0 & 1 & 0 & 0 & 1 &
            $20$ & 0 & 1 & 0 & 0 & 1 & 2 & 1 \\
            \hline
            $7$ & 0 & 0 & 0 & 0 & 1 & 0 & 1 &
            $21$ & 1 & 0 & 1 & 0 & 0 & 0 & 1 \\
            \hline
            $8$ & 0 & 1 & 0 & 1 & 0 & 0 & 1 &
            $22$ & 1 & 0 & 1 & 0 & 0 & 2 & 1 \\
            \hline
            $9$ & 0 & 0 & 0 & 1 & 0 & 0 & 1 &
            $23$ & 1 & 0 & 1 & 0 & 1 & 2 & 1 \\
            \hline
            $10$ & 1 & 0 & 0 & 0 & 0 & 2 & 1 &
            $24$ & 0 & 1 & 0 & 0 & 0 & 2 & 1 \\
            \hline
            $11$ & 1 & 0 & 0 & 0 & 1 & 0 & 1 &
            $25$ & 0 & 1 & 0 & 0 & 1 & 0 & 1 \\
            \hline
            $12$ & 0 & 1 & 0 & 1 & 1 & 2 & 1 &
            $26$ & 1 & 1 & 0 & 0 & 0 & 1 & 1 \\
            \hline
            $13$ & 1 & 0 & 0 & 1 & 1 & 0 & 1 &
            $27$ & 1 & 1 & 0 & 1 & 0 & 1 & $-1$\\
            \hline
            $14$ & 1 & 0 & 0 & 0 & 0 & 0 & 1 &
            $28$ & 1 & 1 & 0 & 1 & 1 & 1 & $-1$ \\
            \hline
        \end{tabular}
        \caption{Parameters defining $\mathcal{M}_{II-B, n}(m)$.}
        \label{tbl:MII-B,params}
    \end{table}

    \begin{table}[H]
        \footnotesize
        \centering
        \begin{tabular}{|c|cc|}
            \hline
            $n$ & $e_1$ & $e_2$ \\
            \hline
            \hline
            1 & 0 & 0 \\
            \hline
            2 & 0 & 1 \\
            \hline
            3 & 1 & 0 \\
            \hline
            4 & 1 & 1 \\
            \hline
            5 & $-1$ & 0 \\
            \hline
        \end{tabular}
        \caption{Parameters defining $\mathcal{M}_{II-C, n}(m)$.}
        \label{tbl:MII-C,params}
    \end{table}

    \begin{table}[H]
        \footnotesize
        \centering
        \begin{tabular}{|c|cccccccc|c|cccccccc|}
            \hline
            $n$ & $e_1$ & $e_2$ & $e_3$ & $e_4$ & $e_5$ & $e_6$ & $e_7$ & $e_8$ &
            $n$ & $e_1$ & $e_2$ & $e_3$ & $e_4$ & $e_5$ & $e_6$ & $e_7$ & $e_8$ \\
            \hline
            \hline
            1 & 0 & 0 & 0 & 1 & 0 & 1 & 0 & 1 &
            13 & 0 & 0 & 0 & 1 & 0 & 0 & 0 & 1 \\
            \hline
            2 & 1 & 1 & 0 & 1 & 0 & 1 & 0 & 1 &
            14 & 1 & 1 & 0 & 1 & 0 & 0 & 0 & 1 \\
            \hline
            3 & 1 & 1 & 0 & 0 & 0 & 1 & 0 & 1 &
            $15^*$ & 1 & 0 & 1 & 1 & 0 & 0 & 0 & 1 \\
            \hline
            4 & 0 & 0 & 0 & 0 & 0 & 1 & 0 & 1 &
            16 & 0 & 0 & 0 & 0 & 1 & 0 & 0 & $-1$ \\
            \hline
            5 & 0 & 0 & 0 & 0 & 0 & 0 & 0 & 1 &
            17 & 1 & 1 & 0 & 0 & 1 & 0 & 0 & 1 \\
            \hline
            6 & 0 & 0 & 0 & 1 & 1 & 0 & 0 & 1 &
            18 & 0 & 0 & 0 & 0 & 0 & 0 & 0 & $-1$ \\
            \hline
            7 & 1 & 1 & 0 & 1 & 1 & 0 & 0 & 1 &
            19 & 1 & 1 & 0 & 0 & 0 & 0 & 0 & 1 \\
            \hline
            8 & 1 & 0 & 1 & 1 & 1 & 0 & 1 & 1 &
            20 & 1 & 1 & 0 & 0 & 0 & 0 & 1 & 1 \\
            \hline
            9 & 0 & 0 & 0 & 0 & 1 & 0 & 1 & $-1$ &
            21 & 0 & 0 & 0 & 1 & 1 & 0 & 1 & 1 \\
            \hline
            $10^*$ & 0 & 1 & 0 & 1 & 1 & 0 & 1 & $-1$ &
            22 & 1 & 1 & 0 & 1 & 1 & 0 & 1 & 1 \\
            \hline
            11 & 0 & 1 & 0 & 1 & 1 & 0 & 0 & $-1$ &
            \multirow{2}{*}{23} & \multirow{2}{*}{0} & \multirow{2}{*}{0} & \multirow{2}{*}{0} & \multirow{2}{*}{0} & \multirow{2}{*}{0} & \multirow{2}{*}{0} & \multirow{2}{*}{1} & \multirow{2}{*}{1} \\
            \cline{1-9}
            12 & 0 & 0 & 0 & 0 & 1 & 0 & 0 & 1 &
            & & & & & & & & \\
            \hline
        \end{tabular}
        \caption{Parameters defining $\mathcal{M}_{II-D, n}(m)$. $\mathcal{M}_{II-D, 15}(m) \simeq \mathcal{M}_{II-D, 10}(m)$.}
        \label{tbl:MII-D,params}
    \end{table}

    \begin{table}[H]
        \footnotesize
        \centering
        \begin{tabular}{|c|cccc|}
            \hline
            $n$ & $e_1$ & $e_2$ & $e_3$ & $e_4$ \\
            \hline
            \hline
            $1$ & $-1$ & 1 & 1 & 0 \\
            \hline
            $2$ & $-1$ & 1 & 3 & 0 \\
            \hline
            $3$ & $-1$ & 1 & 1 & 1 \\
            \hline
            $4$ & 1 & 1 & 2 & 1 \\
            \hline
            $5$ & 1 & 0 & 0 & 0 \\
            \hline
            $6$ & 1 & 1 & 0 & 0 \\
            \hline
            $7$ & 1 & 0 & 2 & 0 \\
            \hline
            $8$ & 1 & 0 & 0 & 1 \\
            \hline
        \end{tabular}
        \caption{Parameters defining $\mathcal{M}_{II-E, n}(m)$.}
        \label{tbl:MII-E,params}
    \end{table}

    \begin{table}[H]
        \footnotesize
        \centering
        \begin{tabular}{|c|ccccc|}
            \hline
            $n$ & $e_1$ & $e_2$ & $e_3$ & $e_4$ & $e_5$ \\
            \hline
            \hline
            $1$ & 0 & 0 & 0 & 0 & 0\\
            \hline
            $2$ & 1 & 0 & 0 & 0 & 0 \\
            \hline
            $3$ & 1 & 0 & 0 & 0 & 1 \\
            \hline
            $4$ & 0 & 1 & 0 & 0 & 0 \\
            \hline
            $5$ & 1 & 1 & 0 & 0 & 0 \\
            \hline
            $6$ & 0 & 1 & 1 & 0 & 0 \\
            \hline
            $7$ & 0 & 0 & 1 & 0 & 0 \\
            \hline
            $8$ & 0 & 0 & 1 & 1 & 0 \\
            \hline
            $9$ & 0 & 1 & 1 & 1 & 0 \\
            \hline
        \end{tabular}
        \caption{Parameters defining $\mathcal{M}_{II-F, n}(m)$.}
        \label{tbl:MII-F,params}
    \end{table}

    \begin{table}[H]
        \footnotesize
        \centering
        \begin{tabular}{|c|cccc|}
            \hline
            $n$ & $e_1$ & $e_2$ & $e_3$ & $e_4$ \\
            \hline
            \hline
            $1$ & 0 & 0 & 0 & 0 \\
            \hline
            $2$ & 0 & 1 & 1 & 0 \\
            \hline
            $3$ & 0 & 0 & 1 & 0 \\
            \hline
            $4$ & 1 & 0 & 0 & 0 \\
            \hline
            $5$ & 1 & 0 & 0 & 1 \\
            \hline
            $6$ & 1 & 0 & 1 & 0 \\
            \hline
            $7$ & 1 & 1 & 0 & 1 \\
            \hline
        \end{tabular}
        \caption{Parameters defining $\mathcal{M}_{II-G, n}(m)$.}
        \label{tbl:MII-G,params}
    \end{table}

    \begin{table}[H]
        \footnotesize
        \centering
        \begin{tabular}{|c|cccccccc|}
            \hline
            $n$ & $e_1$ & $e_2$ & $e_3$ & $e_4$ & $e_5$ & $e_6$ & $e_7$ & $e_8$ \\
            \hline
            \hline
            $1$ & 1 & 1 & 0 & 0 & 0 & 0 & 1 & 0 \\
            \hline
            $2$ & $-1$ & 1 & 1 & 0 & 0 & 1 & 1 & 0 \\
            \hline
            $3$ & $-1$ & 1 & 0 & 0 & 0 & 0 & 1 & 0 \\
            \hline
            $4$ & $-1$ & 1 & 1 & 0 & 0 & 1 & 0 & 0 \\
            \hline
            $5$ & $-1$ & 1 & 0 & 0 & 0 & 0 & 0 & 0 \\
            \hline
            $6$ & $-1$ & 1 & 0 & 0 & 0 & 1 & 0 & 0 \\
            \hline
            $7$ & $-1$ & 1 & 0 & 0 & 1 & 0 & 0 & 0 \\
            \hline
            $8$ & 1 & 1 & 0 & 0 & 0 & 1 & 0 & 0 \\
            \hline
            $9$ & 1 & 1 & 0 & 0 & 0 & 0 & 0 & 0 \\
            \hline
            $10$ & 1 & $-1$ & 0 & 1 & 1 & 0 & 0 & 0 \\
            \hline
            $11$ & 1 & $-1$ & 1 & 1 & 1 & 0 & 0 & 0 \\
            \hline
            $12$ & 1 & 1 & 1 & 1 & 0 & 0 & 0 & 0 \\
            \hline
            $13$ & $-1$ & 1 & 1 & 1 & 0 & 0 & 0 & 0 \\
            \hline
            $14$ & $-1$ & $-1$ & 0 & 1 & 0 & 0 & 0 & 1 \\
            \hline
        \end{tabular}
        \caption{Parameters defining $\mathcal{M}_{III, n}(m)$.}
        \label{tbl:MIII,params}
    \end{table} 
    
    \pagebreak
    
    {
\section{Set of involutions in $\mathcal{G}$}
\label{app:invs}
    \begin{table}[H]
        \footnotesize
        \centering
        \begin{tabular}{|c|c|c|c|c|}
            \hline
            \multirow{2}{*}{$n$} & \multicolumn{2}{c|}{$\text{inv}(\mathcal{N}_{II, n}(m))$} & \multirow{2}{*}{$|\text{inv}(\mathcal{N}_{II, n}(m))|$} & \multirow{2}{*}{$\varphi_X$} \\
            \cline{2-3}
            & \textbf{Central} & \textbf{Non-Central} & & \\
            \hline
            \hline
            \multirow{2}{*}{$1$} & \multirow{5}{*}{$p^{2^{m - 3}}$} & $q$, $p^{2^{m - 3}}q$, $r$, $p^{2^{m - 3}}r$, & \multirow{2}{*}{$1 + 6$} & \multirow{3}{*}{$\varphi_{\{p\}}$} \\
            & & $p^{2^{m - 4}}qr$, $p^{3 \cdot 2^{m - 4}}qr$ & & \\
            \cline{1-1}
            \cline{3-4}

            $2$ & & $q$, $p^{2^{m - 3}}q$, $p^kr$ & $1 + (2 + 2^{m - 3})$ & \\
            \cline{1-1}
            \cline{3-5}

            $3$ & & $q$, $p^{2^{m - 3}}q$, $p^ir$ & $1 + (2 + 2^{m - 2})$ & \multirow{3}{*}{---} \\
            \cline{1-1}
            \cline{3-4}

            $4$ & & $q$, $p^{2^{m - 3}}q$, $p^kr$, $p^iqr$ & $1 + (2 + 2^{m - 3} + 2^{m - 2})$ & \\
            \cline{1-2}
            \cline{3-4}

            $5$ & \multirow{4}{*}{$p^{2^{m - 3}}$, $q$, $p^{2^{m - 3}}q$} & $p^ir$, $p^iqr$ & $3 + 2^{m - 1}$ & \\
            \cline{1-1}
            \cline{3-5}

            $6$ & & $p^kr$, $p^kqr$ & $3 + 2^{m - 2}$ & \multirow{5}{*}{$\varphi_{\{p\}}$} \\
            \cline{1-1}
            \cline{3-4}

            $7$ & & --- & $3 + 0$ & \\
            \cline{1-1}
            \cline{3-4}

            $8$ & & $r$, $p^{2^{m - 3}}r$, $qr$, $p^{2^{m - 3}}qr$ & $3 + 4$ & \\
            \cline{1-4}
            \multirow{2}{*}{$9$} & $p^{2^{m - 3}}$, $q$, $p^{2^{m - 3}}q$ & \multirow{2}{*}{---} & \multirow{2}{*}{$7 + 0$} & \\
            & $r$, $p^{2^{m - 3}}r$, $qr$, $p^{2^{m - 3}}qr$ & & & \\
            \hline
        \end{tabular}
        \caption{The set $\text{inv}(\mathcal{N}_{II, n}(m))$ of involutions in $\mathcal{N}_{II, n}(m)$, where $i = 0, 1, \ldots, 2^{m - 2} - 1 \mod{2^{m - 2}}$, $k = 0, 2, \ldots, 2^{m - 2} - 2 \mod{2^{m - 2}}$. The number of involutions in the fourth column is expressed as the sum of the number of central and non-central involutions. A blank entry in the last column denotes that the group is generated by its involutions.}
        \label{tbl:NII,invslist}
    \end{table}

{
    \LTcapwidth=\textwidth
    \footnotesize
    \begin{longtabu}[h]{|c|c|c|c|c|}
        \hline
        \multirow{2}{*}{$n$} & \multicolumn{2}{c|}{ $\text{inv}(\mathcal{G})$} & \multirow{2}{*}{$|\text{inv}(\mathcal{G})|$} & \multirow{2}{*}{$\varphi_X$} \\
        \cline{2-3}
        & {\textbf{Central}} & {\textbf{Non-Central}} & & \\
        \hline
        \hline
        \multicolumn{5}{|c|}{ $\mathcal{G} = \mathcal{M}_{II-A, n}(m)$} \\
        \hline
        \multirow{2}{*}{$1, 2$} & $p^{2^{m - 4}}$, $q^2$, $p^{2^{m - 4}}q^2$ & \multirow{2}{*}{---} & \multirow{2}{*}{$7 + 0$} & \multirow{9}{*}{$\varphi_{\{q\}}$} \\
        & $r$, $p^{2^{m - 4}}r$, $q^2r$, $p^{2^{m - 4}}q^2r$ & & & \\
        \cline{1-4}
        $3$ & $p^{2^{m - 4}}$, $r$, $p^{2^{m - 4}}r$ & $q^2$, $p^{2^{m - 4}}q^2$, $q^2r$, $p^{2^{m - 4}}q^2r$ & $3 + 4$ & \\
        \cline{1-4}
        \multirow{2}{*}{$4, 5$} & \multirow{2}{*}{$p^{2^{m - 4}}$} & $q^2$, $p^{2^{m - 4}}q^2$, $r$, $p^{2^{m - 4}}r$, & \multirow{2}{*}{1 + 6} & \\
        & & $q^2r$, $p^{2^{m - 4}}q^2r$ & & \\
        \cline{1-4}
        $6 - 11$ & \multirow{9}{*}{$p^{2^{m - 4}}$, $q^2$, $p^{2^{m - 4}}q^2$} & $r$, $p^{2^{m - 4}}r$, $q^2r$, $p^{2^{m - 4}}q^2r$ & $3 + 4$ & \\
        \cline{1-1}
        \cline{3-4}
        $12 - 16$ & & $p^kr$, $p^kq^2r$ & $3 + 2^{m - 3}$ & \\
        \cline{1-1}
        \cline{3-4}
        $17 - 29$ & & --- & $3 + 0$ & \\
        \cline{1-1}
        \cline{3-4}
        $30 - 32$ & & $p^ir$, $p^iq^2r$ & \multirow{2}{*}{$3 + 2^{m - 2}$} & \\
        \cline{1-1}
        \cline{3-3}
        \cline{5-5}
        $33 - 35$ & & $p^kr$, $p^kqr$, $p^kq^2r$, $p^kq^3r$ & & \multirow{3}{*}{$\varphi_{\{p\}}$} \\
        \cline{1-1}
        \cline{3-4}
        \multirow{2}{*}{$36 - 38$} & & $r$, $p^{2^{m - 4}}r$, $qr$, $p^{2^{m - 4}}qr$, & \multirow{2}{*}{$3 + 8$} & \\
        & & $q^2r$, $p^{2^{m - 4}}q^2r$, $q^3r$, $p^{2^{m - 4}}q^3r$ & & \\
        \cline{1-1}
        \cline{3-4}
        \cline{5-5}
        \newpage
        \cline{1-1}
        \cline{3-4}
        \cline{5-5}
        $39$ & & $p^ir$, $p^iqr$, $p^iq^2r$, $p^iq^3r$ & $3 + 2^{m - 1}$ & \multirow{2}{*}{---} \\
        \cline{1-1}
        \cline{3-4}
        $40$ & & $p^ir$, $p^kqr$, $p^iq^2r$, $p^kq^3r$ & $3 + (2^{m - 3} + 2^{m - 2})$ & \\
        \hline
        $41, 42$ & \multirow{7}{*}{$p^{2^{m - 4}}$} & $q^2$, $p^{2^{m - 4}}q^2$, $p^ir$ & $1 + (2 + 2^{m - 3})$ & \multirow{3}{*}{$\varphi_{\{q\}}$} \\
        \cline{1-1}
        \cline{3-4}
        $43$ & & $q^2$, $p^{2^{m - 4}}q^2$, $p^ir$, $p^kq^2r$ & $1 + (2 + 2^{m - 4} + 2^{m - 3})$ & \\
        \cline{1-1}
        \cline{3-4}
        $44$ & & $q^2$, $p^{2^{m - 4}}q^2$, $p^kr$ & $1 + (2 + 2^{m - 4})$ & \\
        \cline{1-1}
        \cline{3-5}
        \multirow{4}{*}{$45, 46$} & & $q^2$, $p^{2^{m - 4}}q^2$, $p^kr$, $r$, & \multirow{4}{*}{$1 + 10$} & \multirow{4}{*}{$\varphi_{\{p\}}$} \\
        & & $p^{2^{m - 4}}r$, $p^{3 \cdot 2^{m - 6}}qr$, $p^{7 \cdot 2^{m - 6}}qr$, & & \\
        & & $p^{2^{m - 5}}q^2r$, $p^{3 \cdot 2^{m - 5}}q^2r$, & & \\
        & & $p^{2^{m - 6}}q^3r$, $p^{5 \cdot 2^{m - 6}}q^3r$  & & \\
        \hline

        \multicolumn{5}{|c|}{ $\mathcal{G} = \mathcal{M}_{II-B, n}(m)$} \\
        \hline
        \multirow{2}{*}{$1, 2$} & $p^{2^{m - 4}}$, $q^2$, $p^{2^{m - 4}}q^2$, & \multirow{2}{*}{---} & \multirow{2}{*}{$7 + 0$} & \multirow{9}{*}{$\varphi_{\{p, q\}}$} \\
        & $r$, $p^{2^{m - 4}}r$, $q^2r$, $p^{2^{m - 4}}q^2r$ & & & \\
        \cline{1-4}
        $3$ & $p^{2^{m - 4}}$, $r$, $p^{2^{m - 4}}r$ & $q^2$, $p^{2^{m - 4}}q^2$, $q^2r$, $p^{2^{m - 4}}q^2r$ & $3 + 4$ & \\
        \cline{1-4}
        \multirow{2}{*}{$4, 5$} & \multirow{2}{*}{$p^{2^{m - 4}}$} & $q^2$, $p^{2^{m - 4}}q^2$, $r$, $p^{2^{m - 4}}r$, & \multirow{2}{*}{$1 + 6$} & \\
        & & $q^2r$, $p^{2^{m - 4}}q^2r$ & & \\
        \cline{1-4}
        $6 - 12$ & \multirow{10}{*}{$p^{2^{m - 4}}$, $q^2$, $p^{2^{m - 4}}q^2$} & $r$, $p^{2^{m - 4}}r$, $q^2r$, $p^{2^{m - 4}}q^2r$ & $3 + 4$ & \\
        \cline{1-1}
        \cline{3-4}
        \multirow{2}{*}{$13 - 15$} & & $r$, $p^{2^{m - 4}}r$, $q^2r$, $p^{2^{m - 4}}q^2r$, & \multirow{2}{*}{$3 + (4 + 2^{m - 3})$} & \\
        & & $p^jqr$, $p^jq^3r$ & & \\
        \cline{1-1}
        \cline{3-4}
        $16 - 18$ & & --- & $3 + 0$ & \\
        \cline{1-1}
        \cline{3-5}
        \multirow{2}{*}{$19, 20$} & & $r$, $p^{2^{m - 4}}r$, $q^2r$, $p^{2^{m - 4}}q^2r$, & \multirow{2}{*}{$3 + (4 + 2^{m - 2})$} & \multirow{2}{*}{---} \\
        & & $p^iqr$, $p^iq^3r$ & & \\
        \cline{1-1}
        \cline{3-5}
        $21 - 23$ & & $p^kqr$, $p^kq^3r$ & $3 + 2^{m - 3}$ & \multirow{7}{*}{$\varphi_{\{p\}}$}  \\
        \cline{1-1}
        \cline{3-4}
        \multirow{2}{*}{$24, 25$} & & $r$, $p^{2^{m - 4}}r$, $q^2r$, $p^{2^{m - 4}}q^2r$, & \multirow{2}{*}{$3 + (4 + 2^{m - 3})$} & \\
        & & $p^kqr$, $p^kq^3r$ & & \\
        \cline{1-4}
        \multirow{3}{*}{$26$} & \multirow{4}{*}{$p^{2^{m - 4}}$} & $q^2$, $p^{2^{m - 4}}q^2$, $r$, & \multirow{3}{*}{$1 + (6 + 2^{m - 3})$} & \\
        & & $p^{2^{m - 4}}r$, $q^2r$, $p^{2^{m - 4}}q^2r$, & & \\
        & & $p^kqr$, $p^kq^3r$ & & \\
        \cline{1-1}
        \cline{3-4}
        $27, 28$ & & $q^2$, $p^{2^{m - 4}}q^2$, $p^{{k_1}}r$, $p^{{k_2}}q^2r$ & $1 + (2 + 2^{m - 4})$ & \\
        \hline

        \multicolumn{5}{|c|}{ $\mathcal{G} = \mathcal{M}_{II-C, n}(m)$} \\
        \hline
        \multirow{2}{*}{$1$} & \multirow{2}{*}{$p^{2^{m - 4}}$, $r$, $p^{2^{m - 4}}r$} & $p^{2^{m - 5}}q^2$, $p^{3 \cdot 2^{m - 5}}q^2$, & \multirow{2}{*}{$3 + 4$} & \multirow{8}{*}{$\varphi_{\{p, q\}}$} \\
        & & $p^{2^{m - 5}}q^2r$, $p^{3 \cdot 2^{m - 5}}q^2r$ & & \\
        \cline{1-4}
        \multirow{2}{*}{$2$} & \multirow{8}{*}{$p^{2^{m - 4}}$} & $p^{2^{m - 5}}q^2$, $p^{3 \cdot 2^{m - 5}}q^2$, $r$, $p^{2^{m - 4}}r$,  & \multirow{4}{*}{$1 + 6$} & \\
        & & $p^{2^{m - 5}}q^2r$, $p^{3 \cdot 2^{m - 5}}q^2r$ & & \\
        \cline{1-1}
        \cline{3-3}
        \multirow{2}{*}{$3$} & & $p^{2^{m - 5}}q^2$, $p^{3 \cdot 2^{m - 5}}q^2$, $r$, $p^{2^{m - 4}}r$, & & \\
        & & $q^2r$, $p^{2^{m - 4}}q^2r$ & & \\
        \cline{1-1}
        \cline{3-4}
        \newpage
        \cline{1-1}
        \cline{3-4}
        \multirow{2}{*}{$4$} & & $p^{2^{m - 5}}q^2$, $p^{3 \cdot 2^{m - 5}}q^2$, $r$, $p^{2^{m - 4}}r$, & \multirow{2}{*}{$1 + (6 + 2^{m - 3})$} & \\
        & & $q^2r$, $p^{2^{m - 4}}q^2r$, $p^jqr$, $p^jq^3r$ & & \\
        \cline{1-1}
        \cline{3-5}
        \multirow{2}{*}{$5$} & & $p^{2^{m - 5}}q^2$, $p^{3 \cdot 2^{m - 5}}q^2$, $r$, $p^{2^{m - 4}}r$, & \multirow{2}{*}{$1 + (6 + 2^{m - 2})$} & \multirow{2}{*}{---} \\
        & & $q^2r$, $p^{2^{m - 4}}q^2r$, $p^iqr$, $p^iq^3r$ & & \\
        \hline

        \multicolumn{5}{|c|}{ $\mathcal{G} = \mathcal{M}_{II-D, n}(m)$} \\
        \hline
        \multirow{3}{*}{$1$} & \multirow{9}{*}{$p^{2^{m - 4}}$} & $q^2$, $p^{2^{m - 4}}q^2$, $r$, $p^{2^{m - 4}}r$, & \multirow{3}{*}{$1 + (6 + 2^{m - 3})$} & \multirow{6}{*}{$\varphi_{\{p, q\}}$} \\
        & &  $q^2r$, $p^{2^{m - 4}}q^2r$, & & \\
        & & $p^{{j_1}}q$, $p^{{j_2}}q^3$, $p^{{j_2}}r$, $p^{{j_1}}q^3r$ & & \\
        \cline{1-1}
        \cline{3-4}
        \multirow{3}{*}{$2$} & & $q^2$, $p^{2^{m - 4}}q^2$, $r$, $p^{2^{m - 4}}r$, & \multirow{3}{*}{$1 + (6 + 2^{m - 3})$} & \\
        & &  $q^2r$, $p^{2^{m - 4}}q^2r$, & & \\
        & & $p^{{j_1}}q$, $p^{{j_2}}q^3$, $p^jqr$ & & \\
        \cline{1-1}
        \cline{3-5}
        \multirow{3}{*}{$3$} & & $q^2$, $p^{2^{m - 4}}q^2$, $r$, $p^{2^{m - 4}}r$, & \multirow{3}{*}{$1 + (6 + 2^{m - 2})$} & \multirow{3}{*}{---} \\
        & &  $q^2r$, $p^{2^{m - 4}}q^2r$, & & \\
        & & $p^{{j_1}}q$, $p^{{j_2}}q^3$, $p^kqr$, $p^iq^3r$ & & \\
        \hline
        \multirow{2}{*}{$4$} & \multirow{2}{*}{$p^{2^{m - 4}}$, $r$, $p^{2^{m - 4}}r$} & $q^2$, $p^{2^{m - 4}}q^2$, $q^2r$, $p^{2^{m - 4}}q^2r$, & \multirow{2}{*}{$3 + (4 + 2^{m - 3})$} & \multirow{13}{*}{$\varphi_{\{p, q\}}$} \\
        & & $p^{{j_1}}q$, $p^{{j_2}}q^3$, $p^{{j_1}}r$, $p^{{j_2}}q^3r$ & & \\
        \cline{1-4}
        \multirow{2}{*}{$5$} & $p^{2^{m - 4}}$, $q^2$, $p^{2^{m - 4}}q^2$, $r$ & \multirow{2}{*}{$p^jq$, $p^jq^3$, $p^jqr$, $p^jq^3r$} & \multirow{2}{*}{$7 + 2^{m - 2}$} & \\
        & $p^{2^{m - 4}}r$, $q^2r$, $p^{2^{m - 4}}q^2r$ & & & \\
        \cline{1-4}
        \multirow{2}{*}{$6, 7$} & \multirow{17}{*}{$p^{2^{m - 4}}$, $q^2$, $p^{2^{m - 4}}q^2$} & $r$, $p^{2^{m - 4}}r$, $q^2r$ $p^{2^{m - 4}}q^2r$, & \multirow{2}{*}{$3 + (4 + 2^{m - 2})$} & \\
        & & $p^jq$, $p^jq^3$, $p^jqr$, $p^jq^3r$ & & \\
        \cline{1-1}
        \cline{3-4}
        $8$ & & --- & $3 + 0$ & \\
        \cline{1-1}
        \cline{3-4}
        $9, 10$ & & $p^kr$, $p^kq^2r$ & $3 + 2^{m - 3}$ & \\
        \cline{1-1}
        \cline{3-4}
        $11$ & & $p^jq$, $p^jq^3$, $p^kr$, $p^kq^2r$ & $3 + 2^{m - 2}$ & \\
        \cline{1-1}
        \cline{3-4}
        \multirow{2}{*}{$12 - 14$} & & $r$, $p^{2^{m - 4}}r$, $q^2r$, $p^{2^{m - 4}}q^2r$, & \multirow{2}{*}{$3 + (4 + 2^{m - 3})$} & \\
        & & $p^jq$, $p^jq^3$ & & \\
        \cline{1-1}
        \cline{3-4}
        $15$ & & $p^jq$, $p^jq^3$ & $3 + 2^{m - 3}$ & \\
        \cline{1-1}
        \cline{3-4}
        $16$ & & $p^jq$, $p^jq^3$, $p^kr$, $p^kq^2r$ & $3 + 2^{m - 2}$ & \\
        \cline{1-1}
        \cline{3-5}
        \multirow{2}{*}{$17$} & & $r$, $p^{2^{m - 4}}r$, $q^2r$ $p^{2^{m - 4}}q^2r$, & \multirow{2}{*}{$3 + (4 + 2^{m - 2})$} & \multirow{5}{*}{---} \\
        & & $p^jq$, $p^jq^3$, $p^kqr$, $p^kq^3r$ & & \\
        \cline{1-1}
        \cline{3-4}
        $18$ & & $p^jq$, $p^jq^3$, $p^ir$, $p^iq^2r$ & $3 + (2^{m - 3} + 2^{m - 2})$ & \\
        \cline{1-1}
        \cline{3-4}
        \multirow{2}{*}{$19$} & & $r$, $p^{2^{m - 4}}r$, $q^2r$, $p^{2^{m - 4}}q^2r$, & \multirow{2}{*}{$3 + (4 + 2^{m - 3} + 2^{m - 2})$} & \\
        & & $p^jq$, $p^jq^3$, $p^iqr$, $p^iq^3r$ & & \\
        \cline{1-1}
        \cline{3-5}
        \multirow{2}{*}{$20$} & & $r$, $p^{2^{m - 4}}r$, $q^2r$, $p^{2^{m - 4}}q^2r$, & \multirow{2}{*}{$3 + (4 + 2^{m - 3})$} & \multirow{5}{*}{$\varphi_{\{p\}}$} \\
        & & $p^kqr$, $p^kq^3r$ & & \\
        \cline{1-1}
        \cline{3-4}
        $21, 22$ & & $r$, $p^{2^{m - 4}}r$, $q^2r$, $p^{2^{m - 4}}q^2r$ & $3 + 4$ & \\
        \cline{1-4}
        \newpage
        \cline{1-4}
        \multirow{2}{*}{$23$} & $p^{2^{m - 4}}$, $q^2$, $p^{2^{m - 4}}q^2, r$ & \multirow{2}{*}{---} & \multirow{2}{*}{$7 + 0$} & \\
        & $p^{2^{m - 4}}r$, $q^2r$, $p^{2^{m - 4}}q^2r$ & & & \\
        \hline

        \multicolumn{5}{|c|}{ $\mathcal{G} = \mathcal{M}_{II-E, n}(m)$} \\
        \hline
        \multirow{2}{*}{$1$} & \multirow{5}{*}{$p^{2^{m - 4}}$} & $q^2$, $p^{2^{m - 4}}q^2$, $p^ir$, & \multirow{2}{*}{$1 +(2 + 2^{m - 4} + 2^{m - 2})$} & \multirow{2}{*}{---} \\
        & & $p^jqr$, $p^jq^2r$, $p^jq^3r$ & & \\
        \cline{1-1}
        \cline{3-5}
        $2$ & & $q^2$, $p^{2^{m - 4}}q^2$, $p^ir$, $p^jq^2r$ & $1 + (2 + 2^{m - 4} + 2^{m - 3})$ & \multirow{7}{*}{$\varphi_{\{q\}}$} \\
        \cline{1-1}
        \cline{3-4}
        $3$ & & $q^2$, $p^{2^{m - 4}}q^2$, $p^kr$ & $1 + (2 + 2^{m - 4})$ & \\
        \cline{1-1}
        \cline{3-4}
        $4$ & & $q^2$, $p^{2^{m - 4}}q^2$, $r$, $q^2r$, $p^{2^{m - 4}}q^2r$ & $1 + 6$ &  \\
        \cline{1-4}
        \multirow{2}{*}{$5$} & $p^{2^{m - 4}}$, $q^2$, $p^{2^{m - 4}}q^2$ & \multirow{2}{*}{---} & \multirow{2}{*}{$7 + 0$} & \\
        & $r$, $p^{2^{m - 4}}r$, $q^2r$, $p^{2^{m - 4}}q^2r$ & & & \\
        \cline{1-4}
        $6$ & $p^{2^{m - 4}}$, $r$, $p^{2^{m - 4}}r$ & $q^2$, $p^{2^{m - 4}}q^2$, $q^2r$, $p^{2^{m - 4}}q^2r$ & \multirow{2}{*}{$3 + 4$} & \\
        \cline{1-3}
        $7, 8$ & $p^{2^{m - 4}}$, $q^2$, $p^{2^{m - 4}}q^2$ & $r$, $p^{2^{m - 4}}r$, $q^2r$, $p^{2^{m - 4}}q^2r$ & & \\
        \hline

        \multicolumn{5}{|c|}{ $\mathcal{G} = \mathcal{M}_{II-F, n}(m)$} \\
        \hline
        \multirow{3}{*}{$1$} & \multirow{8}{*}{$p^{2^{m - 4}}$, $r$, $p^{2^{m - 4}}r$} & $p^{2^{m - 5}}q^2$, $p^{3 \cdot 2^{m - 5}}q^2$, & \multirow{3}{*}{$3 + (4 + 2^{m - 3})$} & \multirow{18}{*}{$\varphi_{\{p, q\}}$} \\
        & & $p^{2^{m - 5}}q^2r$, $p^{3 \cdot 2^{m - 5}}q^2r$, & & \\
        & & $p^{{j_2}}q$, $p^{{j_1}}q^3$, $p^{{j_2}}qr$, $p^{{j_1}}q^3r$ & & \\
        \cline{1-1}
        \cline{3-4}
        \multirow{2}{*}{$2$} & & $p^{2^{m - 5}}q^2$, $p^{3 \cdot 2^{m - 5}}q^2$, & \multirow{2}{*}{$3 + 4$} & \\
        & & $p^{2^{m - 5}}q^2r$, $p^{3 \cdot 2^{m - 5}}q^2r$, & & \\
        \cline{1-1}
        \cline{3-4}
        \multirow{3}{*}{$3$} & & $p^{2^{m - 5}}q^2$, $p^{3 \cdot 2^{m - 5}}q^2$, & \multirow{3}{*}{$3 + (4 + 2^{m - 2})$} & \\
        & & $p^{2^{m - 5}}q^2r$, $p^{3 \cdot 2^{m - 5}}q^2r$, & & \\
        & & $p^jq$, $p^jq^3$, $p^jqr$, $p^jq^3r$ & & \\
        \cline{1-4}
        \multirow{3}{*}{$4$} & \multirow{16}{*}{$p^{2^{m - 4}}$} & $p^{2^{m - 5}}q^2$, $p^{3 \cdot 2^{m - 5}}q^2$, $r$, $p^{2^{m - 4}}r$, & \multirow{5}{*}{$1 + (6 + 2^{m - 3})$} & \\
        & & $p^{2^{m - 5}}q^2r$, $p^{3 \cdot 2^{m - 5}}q^2r$, & & \\
        & & $p^{{j_2}}q$, $p^{{j_1}}q^3$, $p^{{j_1}}qr$, $p^{{j_2}}q^3r$ & & \\
        \cline{1-1}
        \cline{3-3}
        \multirow{2}{*}{$5$} & & $p^{2^{m - 5}}q^2$, $p^{3 \cdot 2^{m - 5}}q^2$, $r$, $p^{2^{m - 4}}r$, & & \\
        & & $p^{2^{m - 5}}q^2r$, $p^{3 \cdot 2^{m - 5}}q^2r$, $p^{j}qr$, $p^jq^3r$ & & \\
        \cline{1-1}
        \cline{3-4}
        \multirow{2}{*}{$6$} & & $p^{2^{m - 5}}q^2$, $p^{3 \cdot 2^{m - 5}}q^2$, $r$, $p^{2^{m - 4}}r$, & \multirow{2}{*}{$1 + (6 + 2^{m - 4})$} & \\
        & & $q^2r$, $p^{2^{m - 4}}q^2r$, $p^{{j_2}}q$, $p^{{j_1}}q^3$ & & \\
        \cline{1-1}
        \cline{3-4}
        \multirow{3}{*}{$7$} & & $p^{2^{m - 5}}q^2$, $p^{3 \cdot 2^{m - 5}}q^2$, $r$, $p^{2^{m - 4}}r$, & \multirow{6}{*}{$1 + (6 + 2^{m - 4} + 2^{m - 3})$} & \\
        & & $q^2r$, $p^{2^{m - 4}}q^2r$, & & \\
        & & $p^{{j_2}}q$, $p^{{j_1}}q^3$, $p^jqr$, $p^jq^3r$, & & \\
        \cline{1-1}
        \cline{3-3}
        \cline{5-5}
        \multirow{3}{*}{$8$} & & $p^{2^{m - 5}}q^2$, $p^{3 \cdot 2^{m - 5}}q^2$, $r$, $p^{2^{m - 4}}r$, & & \multirow{6}{*}{---} \\
        & & $q^2r$, $p^{2^{m - 4}}q^2r$, & & \\
        & & $p^{{j_2}}q$, $p^{{j_1}}q^3$, $p^kqr$, $p^kq^3r$ & & \\
        \cline{1-1}
        \cline{3-4}
        \newpage
        \cline{1-1}
        \cline{3-4}
        \multirow{3}{*}{$9$} & & $p^{2^{m - 5}}q^2$, $p^{3 \cdot 2^{m - 5}}q^2$, $r$, $p^{2^{m - 4}}r$, & \multirow{3}{*}{$1 + (6 + 2^{m - 4} + 2^{m - 2})$} & \\
        & & $q^2r$, $p^{2^{m - 4}}q^2r$, & & \\
        & & $p^{{j_2}}q$, $p^{{j_1}}q^3$, $p^iqr$, $p^iq^3r$ & & \\
        \hline

        \multicolumn{5}{|c|}{ $\mathcal{G} = \mathcal{M}_{II-G, n}(m)$} \\
        \hline
        \multirow{3}{*}{$1$} & \multirow{2}{*}{$p^{2^{m - 4}}$, $p^{2^{m - 4} - 2}q^2$, $p^{2^{m - 3} - 2}q^2$,} & $p^{2^{m - 4} - 1}q$, $p^{2^{m - 3} - 1}q$, $p^{2^{m - 4} - 3}q^3$, & \multirow{3}{*}{$7 + 8$} & \multirow{29}{*}{$\varphi_{\{p, q\}}$} \\
        & \multirow{2}{*}{$r$, $p^{2^{m - 4}}r$, $p^{2^{m - 4} - 2}q^2r$, $p^{2^{m - 3} - 2}q^2r$} & $p^{2^{m - 3} - 3}q^3$, $p^{2^{m - 4} - 1}qr$, $p^{2^{m - 3} - 1}qr$, & & \\
        & & $p^{2^{m - 4} - 3}q^3r$, $p^{2^{m - 3} - 3}q^3r$ & & \\
        \cline{1-4}
        \multirow{4}{*}{$2$} & \multirow{8}{*}{$p^{2^{m - 4}}$, $p^{2^{m - 4} - 2}q^2$, $p^{2^{m - 3} - 2}q^2$} & $p^{2^{m - 4} - 1}q$, $p^{2^{m - 3} - 1}q$, $p^{2^{m - 4} - 3}q^3$, & \multirow{16}{*}{$3 + 12$} & \\
        & & $p^{2^{m - 3} - 3}q^3$, $r$, $p^{2^{m - 4}}r$, & & \\
        & & $p^{2^{m - 4} - 1}qr$, $p^{2^{m - 3} - 1}qr$, $p^{2^{m - 4} - 2}q^2r$, & & \\
        & & $p^{2^{m - 3} - 2}q^2r$, $p^{2^{m - 4} - 3}q^3r$, $p^{2^{m - 3} - 3}q^3r$ & & \\
        \cline{1-1}
        \cline{3-3}
        \multirow{6}{*}{$3$} & & $p^{2^{m - 4} - 1}q$, $p^{2^{m - 3} - 1}q$, $p^{2^{m - 4} - 3}q^3$, & & \\
        & & $p^{2^{m - 3} - 3}q^3$, $r$, $p^{2^{m - 4}}r$, & & \\
        & & $p^{2^{m - 5} - 1}qr$, $p^{3 \cdot 2^{m - 5} - 1}qr$, $p^{2^{m - 4} - 2}q^2r$, & & \\
        & & $p^{2^{m - 3} - 2}q^2r$, $p^{2^{m - 5} - 3}q^3r$, $p^{3 \cdot 2^{m - 5} - 3}q^3r$ & & \\
        \cline{1-3}
        \multirow{4}{*}{$4$} & \multirow{8}{*}{$p^{2^{m - 4}}$, $r$, $p^{2^{m - 4} }r$} & $p^{2^{m - 5} - 1}q$, $p^{3 \cdot 2^{m - 5} - 1}q$, $p^{2^{m - 4} - 2}q^2$, & & \\
        & &  $p^{2^{m - 3} - 2}q^2$, $p^{2^{m - 5} - 3}q^3$, $p^{3 \cdot 2^{m - 5} - 3}q^3$, & & \\
        & & $p^{2^{m - 5} - 1}qr$, $p^{3 \cdot 2^{m - 5} - 1}qr$, $p^{2^{m - 4} - 2}q^2r$, & & \\
        & & $p^{2^{m - 3} - 2}q^2r$, $p^{2^{m - 5} - 3}q^3r$, $p^{3 \cdot 2^{m - 5} - 3}q^3r$ & & \\
        \cline{1-1}
        \cline{3-3}
        \multirow{4}{*}{$5$} & & $p^{3 \cdot 2^{m - 6} - 1}q$, $p^{7 \cdot 2^{m - 6} - 1}q$, $p^{2^{m - 4} - 2}q^2$, & & \\
        & & $p^{2^{m - 3} - 2}q^2$, $p^{2^{m - 6} - 3}q^3$, $p^{5 \cdot 2^{m - 6} - 3}q^3$, & & \\
        & & $p^{3 \cdot 2^{m - 6} - 1}qr$, $p^{7 \cdot 2^{m - 6} - 1}qr$, $p^{2^{m - 4} - 2}q^2r$, & & \\
        & & $p^{2^{m - 3} - 2}q^2r$, $p^{2^{m - 6} - 3}q^3r$, $p^{5 \cdot 2^{m - 6} - 3}q^3r$ & & \\
        \cline{1-4}
        \multirow{5}{*}{$6$} & \multirow{10}{*}{$p^{2^{m - 4}}$} & $p^{2^{m - 5} - 1}q$, $p^{3 \cdot 2^{m - 5} - 1}q$, $p^{2^{m - 4} - 2}q^2$, & \multirow{10}{*}{$1 + 14$} & \\
        & & $p^{2^{m - 3} - 2}q^2$, $p^{2^{m - 5} - 3}q^3$, $p^{3 \cdot 2^{m - 5} - 3}q^3$, & & \\
        & & $r$, $p^{2^{m - 4}}r$, $p^{2^{m - 4} - 1}qr$, & & \\
        & & $p^{2^{m - 3} - 1}qr$, $p^{2^{m - 4} - 2}q^2r$, $p^{2^{m - 3} - 2}q^2r$, & & \\
        & & $p^{2^{m - 4} - 3}q^3r$, $p^{2^{m - 3} - 3}q^3r$ & & \\
        \cline{1-1}
        \cline{3-3}
        \multirow{5}{*}{$7$} & & $p^{3 \cdot 2^{m - 6} - 1}q$, $p^{7 \cdot 2^{m - 6} - 1}q$, $p^{2^{m - 4} - 2}q^2$, & & \\
        & & $p^{2^{m - 3} - 2}q^2$, $p^{2^{m - 6} - 3}q^3$ , $p^{5 \cdot 2^{m - 6} - 3}q^3$, & & \\
        & & $r$, $p^{2^{m - 4}}r$, $p^{2^{m - 6} - 1}qr$, & & \\
        & & $p^{5 \cdot 2^{m - 6} - 1}qr$, $p^{2^{m - 4} - 2}q^2r$, $p^{2^{m - 3} - 2}q^2r$, & & \\
        & & $p^{3 \cdot 2^{m - 6} - 3}q^3r$, $p^{7 \cdot 2^{m - 6} - 3}q^3r$ & & \\
        \hline

        \newpage
        \hline
        \multicolumn{5}{|c|}{ $\mathcal{G} = \mathcal{M}_{III, n}(m)$} \\
        \hline
        \multirow{4}{*}{$1$} & \multirow{10}{*}{$p^{2^{m - 4}}$, $q$, $p^{2^{m - 4}}q$} & $r$, $p^{2^{m - 4}}r$, $qr$,  $p^{2^{m - 4}}qr$, & \multirow{4}{*}{$3 + 12$} & \multirow{6}{*}{$\varphi_{\{p\}}$} \\
        & & $s$, $p^{2^{m - 4}}s$, $qs$, $p^{2^{m - 4}}qs$, & & \\
        & & $p^{2^{m - 5}}rs$, $p^{2^{m - 5}}qrs$, & & \\
        & & $p^{3 \cdot 2^{m - 5}}rs$, $p^{3 \cdot 2^{m - 5}}qrs$ & & \\
        \cline{1-1}
        \cline{3-4}
        \multirow{2}{*}{$2$} & & $r$, $p^{2^{m - 4}}r$, $qr$, $p^{2^{m - 4}}qr$, & \multirow{2}{*}{$3 + (4 + 2^{m - 3})$} & \\
        & & $p^ks$, $p^kqs$ & & \\
        \cline{1-1}
        \cline{3-5}
        \multirow{2}{*}{$3$} & & $r$, $p^{2^{m - 4}}r$, $qr$, $p^{2^{m - 4}}qr$, & \multirow{2}{*}{$3 + (4 + 2^{m - 2})$} & \multirow{5}{*}{---} \\
        & & $p^is$, $p^iqs$ & & \\
        \cline{1-1}
        \cline{3-4}
        \multirow{2}{*}{$4$} & & $r$, $p^{2^{m - 4}}r$, $qr$, $p^{2^{m - 4}}qr$, & \multirow{2}{*}{$3 + (4 + 2^{m - 3} + 2^{m - 2})$} & \\
        & & $p^ks$, $p^kqs$, $p^irs$, $p^iqrs$ & & \\
        \cline{1-2}
        \cline{3-4}
        $5$ & & $p^is$, $p^iqs$, $p^irs$, $p^iqrs$ & $7 + 2^{m - 1}$ & \\
        \cline{1-1}
        \cline{3-5}
        $6$ & \multirow{2}{*}{$p^{2^{m - 4}}$, $q$, $p^{2^{m - 4}}q$, $r$, $qr$,} & $p^ks$, $p^kqs$, $p^krs$, $p^kqrs$ & $7 + 2^{m - 2}$ & \multirow{8}{*}{$\varphi_{\{p\}}$} \\
        \cline{1-1}
        \cline{3-4}
        $7$ & \multirow{2}{*}{$p^{2^{m - 4}}r$, $p^{2^{m - 4}}qr$} & --- & $7 + 0$ & \\
        \cline{1-1}
        \cline{3-4}
        \multirow{2}{*}{$8$} & & $s$, $p^{2^{m - 4}}s$, $qs$, $p^{2^{m - 4}}qs$, $rs$, & \multirow{2}{*}{$7 + 8$} & \\
        & & $p^{2^{m - 4}}rs$, $qrs$, $p^{2^{m - 4}}qrs$ & & \\
        \cline{1-4}
        \multirow{4}{*}{$9$} & $p^{2^{m - 4}}$, $q$, $p^{2^{m - 4}}q$, $r$, $p^{2^{m - 4}}r$, & \multirow{4}{*}{---} & \multirow{4}{*}{$15 + 0$} & \\
        & $qr$, $p^{2^{m - 4}}qr$, $s$, $p^{2^{m - 4}}s$, & & & \\
        & $qs$, $p^{2^{m - 4}}qs$, $rs$, $p^{2^{m - 4}}rs$, & & & \\
        & $qrs$, $p^{2^{m - 4}}qrs$ & & & \\
        \hline
        \multirow{2}{*}{$10$} & \multirow{13}{*}{$p^{2^{m - 4}}$} & $q$, $p^{2^{m - 4}}q$, $p^{2^{m - 5}}s$, $p^{3 \cdot 2^{m - 5}}s$, & \multirow{2}{*}{$1 + (6 + 2^{m - 3} + 2^{m - 2})$} & \multirow{5}{*}{---} \\
        & & $qs$, $p^{2^{m - 4}}qs$, $p^ir$, $p^iqr$, $p^iqrs$ & & \\
        \cline{1-1}
        \cline{3-4}
        \multirow{3}{*}{$11$} & & $q$, $p^{2^{m - 4}}q$, $p^{2^{m - 5}}s$, $p^{3 \cdot 2^{m - 5}}s$, & \multirow{3}{*}{$1 + (6 + 2^{m - 2})$} & \\
        & & $qs$, $p^{2^{m - 4}}qs$, $p^kr$, $p^kqr$, & & \\
        & & $p^jrs$, $p^kqrs$ & & \\
        \cline{1-1}
        \cline{3-5}
        \multirow{4}{*}{$12$} & & $q$, $p^{2^{m - 4}}q$, $r$, $p^{2^{m - 4}}r$, & \multirow{4}{*}{$1 + 14$} & \multirow{4}{*}{$\varphi_{\{p\}}$} \\
        & & $qr$, $p^{2^{m - 4}}qr$, $s$, $p^{2^{m - 4}}s$, & & \\
        & & $p^{2^{m - 5}}qs$, $p^{3 \cdot 2^{m - 5}}qs$, $rs$, $p^{2^{m - 4}}rs$, & & \\
        & & $p^{2^{m - 5}}qrs$, $p^{3 \cdot 2^{m - 5}}qrs$ & & \\
        \cline{1-1}
        \cline{3-5}
        \multirow{2}{*}{$13$} & & $q$, $p^{2^{m - 4}}q$, $r$, $p^{2^{m - 4}}r$, $qr$, & \multirow{2}{*}{$1 + (6 + 2^{m - 2})$} & \multirow{4}{*}{---} \\
        & & $p^{2^{m - 4}}qr$, $p^is$, $p^krs$, $p^jqrs$ & & \\
        \cline{1-1}
        \cline{3-4}
        \multirow{2}{*}{$14$} & & $q$, $p^{2^{m - 4}}q$, $p^is$, $p^{2^{m - 5}}rs$,  & \multirow{2}{*}{$1 + (6 + 2^{m - 3})$} & \\
        & & $p^{3 \cdot 2^{m - 5}}rs$, $qrs$, $p^{2^{m - 4}}qrs$ & & \\
        \hline
        \caption{ The set $\text{inv}(\mathcal{G})$ of involutions in $\mathcal{G}$, where
        $\mathcal{G} = \mathcal{M}_{II-A, n}(m)$,
        $\mathcal{M}_{II-B, n}(m)$,
        $\mathcal{M}_{II-C, n}(m)$,
        $\mathcal{M}_{II-D, n}(m)$,
        $\mathcal{M}_{II-E, n}(m)$,
        $\mathcal{M}_{II-F, n}(m)$,
        $\mathcal{M}_{II-G, n}(m)$,
        $\mathcal{M}_{III, n}(m)$, and
        $i = 0, 1, \ldots, 2^{m - 3} - 1 \mod{2^{m - 3}}$,
        $j = 1, 3, \ldots, 2^{m - 3} - 1 \mod{2^{m - 3}}$,
        $j_1 = 1, 5, \ldots, 2^{m - 3} - 3 \mod{2^{m - 3}}$,
        $j_2 = 3, 7, \ldots, 2^{m - 3} - 1 \mod{2^{m - 3}}$,
        $k = 0, 2, \ldots, 2^{m - 3} - 2 \mod{2^{m - 3}}$,
        $k_1 = 0, 4, \ldots, 2^{m - 3} - 4 \mod{2^{m - 3}}$,
        $k_2 = 2, 6, \ldots, 2^{m - 3} - 2 \mod{2^{m - 3}}$.
        The number of involutions in the fourth column is expressed as the sum of the number of central and non-central involutions. A blank entry in the last column denotes that the group is generated by its involutions.}
        \label{tbl:MII,MIII,invslist}
     \end{longtabu}
}
} 

    \pagebreak

    \section{Pairs of non-commuting involutory coset representatives of $Z(\mathcal{G})$}
\label{app:noncomminvs}

    \begin{table}[H]
        \centering
        \footnotesize
        \begin{tabu}{|l|l|l|}
            \hline
            \rowfont[r]{} $\mathcal{G}$ & $t_0$ & $t_1$ \\
            \hline
            \hline
            \multirow{2}{*}{$\mathcal{N}_{II, 3}(m)$}
            & $q$ & $p^ir$ \\
            \cline{2-3}
            & $p^ir$ & $p^{i'}r$ \\
            \hline

            \multirow{4}{*}{$\mathcal{N}_{II, 4}(m)$}
            & $q$ & $p^jqr$ \\
            \cline{2-3}
            & $p^kr$ & $p^{k'}r$, $p^jqr$, $p^{k'}qr$ \\
            \cline{2-3}
            & $p^jqr$ & $p^{j'}qr$ \\
            \cline{2-3}
            & $p^kqr$ & $p^{k'}qr$ \\
            \hline

            \multirow{1}{*}{$\mathcal{N}_{II, 5}(m)$}
            & $p^ir$ & $p^{i'}r$ \\
            \hline
        \end{tabu}
        \caption{The pairs of non-commuting involutory coset representatives $t_0$, $t_1$ of $Z(\mathcal{G})$, where
        $\mathcal{G} = \mathcal{N}_{II, 3}(m)$,
        $\mathcal{N}_{II, 4}(m)$,
        $\mathcal{N}_{II, 5}(m)$, and
        $i, i' = 0, 1, \ldots, 2^{m - 2} - 1$,
        $i' \neq i \mod{2^{m - 3}}$,
        $j, j' = 1, 3, \ldots, 2^{m - 2} - 1$,
        $j' \neq j \mod{2^{m - 3}}$,
        $k, k' = 0, 2, \ldots, 2^{m - 2} - 2$,
        $k' \neq k \mod{2^{m - 3}}$.}
        \label{tbl:N_3,4,5,noncomminvs}
    \end{table}

{
    \LTcapwidth=\textwidth
    \footnotesize
    \begin{longtabu}[h]{|l|l|l|}
        \hline
         \rowfont[r]{} $\mathcal{G}$ & $t_0$ & $t_1$ \\
         \hline
         \hline
         \multirow{2}{*}{$\mathcal{M}_{II-A, 39}(m)$}
         & $p^ir$ & $p^{i'}r$, $p^{\bar{i}}qr$ \\
         \cline{2-3}
         & $p^iqr$ & $p^{i'}qr$ \\
         \hline

         \multirow{2}{*}{$\mathcal{M}_{II-A, 40}(m)$}
         & $p^ir$ & $p^{i'}r$, $p^kqr$ \\
         \cline{2-3}
         & $p^kqr$ & $p^{k'}qr$ \\
         \hline

         \multirow{1}{*}{$\mathcal{M}_{II-B, 19}(m)$,}
         & $r$ & $p^iqr$ \\
         \cline{2-3}
         \multirow{1}{*}{$\mathcal{M}_{II-B, 20}(m)$}
         & $p^iqr$ & $p^{i'}qr$ \\
         \hline

         \multirow{5}{*}{$\mathcal{M}_{II-C, 5}(m)$}
         & $p^{2^{m - 5}}q^2$ & $r$, $q^2r$ \\
         \cline{2-3}
         & $r$ & $q^2r$, $p^iqr$, $p^iq^3r$ \\
         \cline{2-3}
         & $q^2r$ & $p^iqr$, $p^iq^3r$ \\
         \cline{2-3}
         & $p^iqr$ & $p^{i'}qr$, $p^{i''}q^3r$ \\
         \cline{2-3}
         & $p^iq^3r$ & $p^{i'}q^3r$ \\
         \hline

         \multirow{3}{*}{$\mathcal{M}_{II-D, 17}(m)$}
         & $r$ & $p^jq$, $p^kqr$  \\
         \cline{2-3}
         & $p^jq$ & $p^{j'}q$, $p^kqr$ \\
         \cline{2-3}
         & $p^kqr$ & $p^{k'}qr$ \\
         \hline

         \multirow{2}{*}{$\mathcal{M}_{II-D, 18}(m)$}
         & $p^jq$ & $p^{j'}q$, $p^ir$  \\
         \cline{2-3}
         & $p^ir$ & $p^{i'}r$  \\
         \hline

         \multirow{6}{*}{$\mathcal{M}_{II-E, 1}(m)$}
         & $q^2$ & $p^kr$ \\
         \cline{2-3}
         & $p^jr$ & $p^{j'}r$, $p^{\bar{j}}qr$, $p^{j'}q^2r$, $p^{\bar{j}}q^3r$ \\
         \cline{2-3}
         & $p^kr$ & $p^{k'}r$, $p^jqr$, $p^jq^2r$, $p^jq^3r$ \\
         \cline{2-3}
         & $p^jqr$ & $p^{j'}qr$, $p^{\bar{j}}q^2r$, $p^{j'}q^3r$ \\
         \cline{2-3}
         & $p^jq^2r$ & $p^{j'}q^2r$, $p^{\bar{j}}q^3r$ \\
         \cline{2-3}
         & $p^jq^3r$ & $p^{j'}q^3r$ \\
         \hline

         \multirow{7}{*}{$\mathcal{M}_{II-F, 8}(m)$}
         & $p^{2^{m - 5}}q^2$ & $r$, $q^2r$, $p^{{j_2}}q$, $p^{{j_1}}q^3$ \\
         \cline{2-3}
         & $r$ & $q^2r$, $p^{{j_2}}q$, $p^{{j_1}}q^3$, $p^kqr$, $p^kq^3r$ \\
         \cline{2-3}
         & $q^2r$ & $p^{{j_2}}q$, $p^{{j_1}}q^3$, $p^kqr$, $p^kq^3r$ \\
         \cline{2-3}
         & $p^{{j_2}}q$ & $p^{{j_2'}}q$, $p^{{j_1}}q^3$, $p^kqr$, $p^kq^3r$ \\
         \cline{2-3}
         & $p^{{j_1}}q^3$ & $p^{{j_1'}}q^3$, $p^kqr$, $p^kq^3r$ \\
         \cline{2-3}
         & $p^kqr$ & $p^{k'}qr$, $p^{k''}q^3r$ \\
         \cline{2-3}
         & $p^kq^3r$ & $p^{k'}q^3r$ \\
         \hline

         \multirow{11}{*}{$\mathcal{M}_{II-F, 9}(m)$}
         & $p^{2^{m - 5}}q^2$ & $r$, $q^2r$, $p^{{j_2}}q$, $p^{{j_1}}q^3$ \\
         \cline{2-3}
         & $r$ & $q^2r$, $p^kqr$, $p^{{j_1}}qr$, $p^kq^3r$, $p^{{j_2}}q^3r$ \\
         \cline{2-3}
         & $q^2r$ & $p^{{j_2}}qr$, $p^kqr$, $p^kq^3r$, $p^{{j_1}}q^3r$ \\
         \cline{2-3}
         & $p^{{j_2}}q$ & $p^{{j_2'}}q$, $p^{{j_1}}q^3$, $p^{{j_2'}}qr$, $p^{{j_2'}}q^3r$ \\
         \cline{2-3}
         & $p^{{j_1}}q^3$ & $p^{{j_1'}}q^3$, $p^{{j_1'}}qr$, $p^{{j_1'}}q^3r$ \\
         \cline{2-3}
         & $p^{{j_1}}qr$ & $p^{{j_1'}}qr$, $p^{{j_1''}}q^3r$ \\
         \cline{2-3}
         & $p^{{j_2}}qr$ & $p^{{j_2'}}qr$, $p^{{j_2''}}q^3r$ \\
         \cline{2-3}
         & $p^kqr$ & $p^{k'}qr$, $p^{k''}q^3r$ \\
         \cline{2-3}
         & $p^{{j_1}}q^3r$ & $p^{{j_1'}}q^3r$ \\
         \cline{2-3}
         & $p^{{j_2}}q^3r$ & $p^{{j_2'}}q^3r$ \\
         \cline{2-3}
         & $p^kq^3r$ & $p^{k'}q^3r$ \\
         \hline

         \multirow{6}{*}{$\mathcal{M}_{III, 10}(m)$}
         & $q$ & $p^{2^{m - 5}}s$, $qs$, $p^iqrs$ \\
         \cline{2-3}
         & $p^{2^{m - 5}}s$ & $qs$, $p^ir$ \\
         \cline{2-3}
         & $qs$ & $p^iqr$ \\
         \cline{2-3}
         & $p^ir$ & $p^{i'}r$, $p^{i'}qr, p^{i'}qrs$ \\
         \cline{2-3}
         & $p^iqr$ & $p^{i'}qr$, $p^{i''}qrs$ \\
         \cline{2-3}
         & $p^iqrs$ & $p^{i'}qrs$ \\
         \hline

         \multirow{7}{*}{$\mathcal{M}_{III, 11}(m)$}
         & $q$ & $p^{2^{m - 5}}s$, $qs$, $p^kqrs$ \\
         \cline{2-3}
         & $p^{2^{m - 5}}s$ & $qs$, $p^kr$ \\
         \cline{2-3}
         & $qs$ & $p^kqr$  \\
         \cline{2-3}
         & $p^kr$ & $p^{k'}r$, $p^{k'}qr$, $p^{k'}qrs$  \\
         \cline{2-3}
         & $p^kqr$ & $p^{k'}qr$, $p^{k''}qrs$ \\
         \cline{2-3}
         & $p^jrs$ & $p^{j'}rs$, $p^kqrs$ \\
         \cline{2-3}
         & $p^kqrs$ & $p^{k'}qrs$ \\
         \hline

         \multirow{7}{*}{$\mathcal{M}_{III, 13}(m)$}
         & $q$ & $p^is$, $p^krs$, $p^jqrs$ \\
         \cline{2-3}
         & $r$ & $p^js$, $p^jqrs$ \\
         \cline{2-3}
         & $qr$ & $p^ks$, $p^krs$ \\
         \cline{2-3}
         & $p^js$ & $p^{j'}s$, $p^krs$, $p^{j'}qrs$ \\
         \cline{2-3}
         & $p^ks$ & $p^{k'}s$, $p^{k'}rs$, $p^{j}qrs$  \\
         \cline{2-3}
         & $p^krs$ & $p^{k'}rs$, $p^jqrs$ \\
         \cline{2-3}
         & $p^jqrs$ & $p^{j'}qrs$ \\
         \hline

         \multirow{3}{*}{$\mathcal{M}_{III, 14}(m)$}
         & $q$ & $p^is$, $p^{2^{m - 5}}rs$, $qrs$ \\
         \cline{2-3}
         & $p^is$ & $p^{i'}s$, $p^{2^{m - 5}}rs$, $qrs$ \\
         \cline{2-3}
         & $p^{2^{m - 5}}rs$ & $qrs$ \\
         \hline
         \caption{The pairs of non-commuting involutory coset representatives $t_0$, $t_1$ of $Z(\mathcal{G})$, where $\mathcal{G} = \mathcal{M}_{II-A, 39}(m)$,
         $\mathcal{M}_{II-A, 40}(m)$,
         $\mathcal{M}_{II-B, 19}(m)$,
         $\mathcal{M}_{II-B, 20}(m)$,
         $\mathcal{M}_{II-C, 5}(m)$,
         $\mathcal{M}_{II-D,17}(m)$,
         $\mathcal{M}_{II-D, 18}(m)$,
         $\mathcal{M}_{II-E, 1}(m)$,
         $\mathcal{M}_{II-F, 8}(m)$,
         $\mathcal{M}_{II-F, 9}(m)$,
         $\mathcal{M}_{III, 10}(m)$,
         $\mathcal{M}_{III, 11}(m)$,
         $\mathcal{M}_{III, 13}(m)$,
         $\mathcal{M}_{III, 14}(m)$, and
         $i, \bar{i}, i', i'' = 0, 1, \ldots, 2^{m - 3} - 1$,
         $i' \neq i \mod{2^{m - 4}}$,
         $i'' \neq i + 2^{m - 5} \mod{2^{m - 4}}$,
         $j, \bar{j}, j' = 1, 3, \ldots, 2^{m - 3} - 1$,
         $j' \neq j \mod{2^{m - 4}}$,
         $j_1, j_1', j_1'' = 1, 5, \ldots, 2^{m - 3} - 3$,
         $j_1' \neq j_1 \mod{2^{m - 4}}$,
         $j_1'' \neq j_1 + 2^{m - 5} \mod{2^{m - 4}}$,
         $j_2, j_2', j_2'' = 3, 7, \ldots, 2^{m - 3} - 1$,
         $j_2' \neq j_2 \mod{2^{m - 4}}$,
         $j_2'' \neq j_2 + 2^{m - 5} \mod{2^{m - 4}}$,
         $k, k', k'' = 0, 2, \ldots, 2^{m - 3} - 2$,
         $k' \neq k \mod{2^{m - 4}}$,
         $k'' \neq k + 2^{m - 5} \mod{2^{m - 4}}$.}
         \label{tbl:M_II-A,39,40,II-B,19,20,II-C,5,II-D,17,18,II-E,1,II-F,8,II-F,9,M_III,10,11,13,14,noncomminvs}
     \end{longtabu}
}     
    
    \pagebreak
    
    \section*{Acknowledgments}
    It is an honor to acknowledge my dissertation advisers Prof. Ken-ichi Shinoda and Assoc. Prof. Yasushi Gomi of Sophia University and Prof. Ma. Louise Antonette N. De Las Pe\~{n}as of Ateneo de Manila University for their expertise and invaluable advice during the conduct of this research. It is a pleasure to thank Sophia University, Tokyo, Japan for allowing me to use the university's research facilities during my one-year research stay in Tokyo. I am equally indebted to the Commission on Higher Education Faculty Development Program (CHED FDP) of the Philippine Government for the PhD Sandwich Program Abroad scholarship.


\begin{thebibliography}{09}
    \bibitem{Berkovich2008}
        Y. Berkovich,
        Groups of Prime Power Order, Volume 1,
        de Gruyter Expositions in Mathematics 46,
        Walter de Gruyter, Germany, 2008.

    \bibitem{Bai1985}
        S. Bai,
        \textit{A classification of 2-groups with a cyclic subgroup of index $2^2$},
        J. Heilongjiang Univ. Natur. Sci. \textbf{2} (1985) 74–85.

    \bibitem{Babb1910}
        M.J. Babb,
        \textit{The second category of the groups of order $2^m$ which contain self-conjugate cyclic sub-groups of order $2^{m - 4}$},
        PhD Dissertation, University of Pennsylvania, 1910.

    \bibitem{Burnside1911}
        W. Burnside,
        \textit{Theory of Groups of Finite Order, 2nd Ed},
        Cambridge University Press, UK, 1911.


    \bibitem{ConderOliveros2013}
        M. Conder, D. Oliveros,
        The intersection condition for regular polytopes,
        J. Combin. Theory Ser. A 120 (2013) 1291-1304.

    \bibitem{FernandesLeemans2011}
        M.E. Fernandes, D. Leemans,
        Polytopes of high rank for the symmetric groups,
        Adv. Math. 228 (2011) 3207–3222.

    \bibitem{FernandesLeemansMixer2012}
        M.E. Fernandes, D. Leemans, M. Mixer,
        Polytopes of high rank for the alternating groups,
        J. Combin. Theory Ser. A 119 (2012) 42–56.

    \bibitem{Finkel1906}
        B.F. Finkel,
        \textit{Determination of all groups of order $2^m$ which contain cyclic self-conjugate sub-groups of order $2^{m - 4}$ and whose generating operations correspond to the partitions $(m - 4, 4)$, $(m - 4, 3, 1)$},
        PhD Dissertation, University of Pennsylvania, 1906.

    \bibitem{GAP2015}
        The GAP~Group,
        GAP -- Groups, Algorithms, and Programming, Version 4.7.8,
        \verb+http://www.gap-system.org+, 2015.

    \bibitem{Magma2015}
        Computational Algebra Group, School of Mathematics and Statistics, University of Sydney,
        \textit{Magma Computational Algebra System},
        \verb+http://magma.maths.usyd.edu.au/magma/+, 2015.

    \bibitem{Hartley2006}
        M.I. Hartley,
        An atlas of small regular abstract polytopes,
        Period. Math. Hungar. 53(1–2) (2006) 149–156.

    \bibitem{HartleyHulpke2010}
        M.I. Hartley, A. Hulpke,
        Polytopes derived from sporadic simple groups,
        Contributions to discrete matematics 5(2) (2010) 106-118.

    \bibitem{Humphreys1990}
        J. Humphreys,
        Reflection groups and Coxeter groups,
        Cambridge Studies in Advanced Mathematics 29,
        Cambridge University Press, Cambridge, 1990.

    \bibitem{JohnsonWeiss1999}
        N.W. Johnson, A.I. Weiss,
        Quaternionic modular groups,
        Linear Algebra Appl. 295 (1999) 159-189.

    \bibitem{Leemans2006}
        D. Leemans,
        Almost simple groups of Suzuki type acting on polytopes,
        Proc. Amer. Math. Soc. 134(12) (2006) 3649–3651.

    \bibitem{LeemansVauthier2006}
        D. Leemans, L. Vauthier,
        An atlas of abstract regular polytopes for small groups,
        Aequationes Math. 72 (2006) 313–320.

    \bibitem{McKelden1906}
        A.M. McKelden,
        Groups of order $2^m$ that contain cyclic subgroups of order $2^{m - 3}$,
        Amer. Math. Monthly 13(6/7) (1906) 121-136d.

    \bibitem{McMullenSchulte1992}
        P. McMullen, E. Schulte,
        Locally toroidal regular polytopes of rank 4,
        Comment. Math. Helvetici 67 (1992) 77-118.

    \bibitem{McMullenSchulte2002}
        P. McMullen, E. Schulte,
        Abstract Regular Polytopes,
        Encyclopedia of Mathematics and Its Applications,
        Cambridge University Press, Cambridge, 2002.

    \bibitem{Miller1901}
        G.A. Miller,
        Determination of all the groups of order $p^m$ which contain the abelian group of type $(m - 2, 1)$, $p$ being any prime,
        Trans. Amer. Math. Soc. 2 (1901) 259-272.

    \bibitem{Miller1902}
        G.A. Miller,
        On the groups of order $p^m$ which contains operators of order $p^{m - 2}$,
        Trans. Amer. Math. Soc. 3(4) (1902) 383-387.

    \bibitem{Monson1995}
        B. Monson,
        Polytopes related to the Picard group,
        Linear Algebra Appl. 218 (1995) 185-204.

    \bibitem{Ninomiya1994}
        Y. Ninomiya,
        Finite $p$-groups with cyclic subgroups of index $p^2$,
        Math. J. Okayama Univ. 36 (1994) 1-21.

    \bibitem{SchulteWeiss2006}
        E. Schulte, A.I. Weiss,
        Problems on polytopes, their groups, and realizations,
        Periodica Math. Hungar. 53(1–2) (2006) 231–255.

    \bibitem{ZhangLi2012}
        Q. Zhang, P. Li,
        Finite $p$-groups with a cyclic subgroup of index $p^3$,
        Journal of Mathematical Research with Applications 32(5) (2012) 505-529.
\end{thebibliography}
\end{document}